\documentclass[12pt,reqno]{amsart}
\usepackage{amsmath,amssymb,amsfonts}
\usepackage{mathrsfs}
\usepackage[abbrev]{amsrefs}
\usepackage[pdftex]{graphicx}

\usepackage{latexsym}
\usepackage{delarray}

\newtheorem{thm}{Theorem}[section]
\newtheorem{cor}[thm]{Corollary}
\newtheorem{prop}[thm]{Proposition}
\newtheorem{lem}[thm]{Lemma}
\theoremstyle{definition}
\newtheorem{rem}[thm]{Remark}
\newtheorem{definition}[thm]{Definition}

\numberwithin{equation}{section}
\numberwithin{thm}{section}
\makeatletter
\newcommand{\vertiii}[1]{{\left\vert\kern-0.25ex\left\vert\kern-0.25ex\left\vert #1
    \right\vert\kern-0.25ex\right\vert\kern-0.25ex\right\vert}}

\title[Dissipation and blow up for semilinear heat equations]{
Global well-posedness, dissipation and blow up for semilinear heat equations in energy spaces associated with self-adjoint operators}
\author[
M. Ikeda and K. Taniguchi]{Masahiro Ikeda and Koichi Taniguchi}
\address{
Masahiro Ikeda \endgraf
Department of Mathematics \endgraf
Faculty of Science and Technology \endgraf
Keio University \endgraf
3-14-1 Hiyoshi, Kohoku-ku, Yokohama, 223-8522, Japan/\endgraf
Center for Advanced Intelligence Project \endgraf
RIKEN, Japan
}
\email{masahiro.ikeda@keio.jp/masahiro.ikeda@riken.jp}
\address{
Koichi Taniguchi \endgraf
Graduate School of Mathematics \endgraf
Nagoya University \endgraf
Furocho, Chikusaku, Nagoya 464-8602, Japan}
\email{koichi-t@math.nagoya-u.ac.jp}
\date{\today}


\begin{document}


\begin{abstract}
The purpose in this paper is to determine the global behavior of solutions to the initial-boundary value problems for
energy-subcritical and critical semilinear heat equations by initial data with lower energy than the mountain pass level 
in energy spaces associated with self-adjoint operators satisfying Gaussian upper bounds. 
Our self-adjoint operators include the Dirichlet Laplacian on an open set, Robin Laplacian on an exterior domain, and Schr\"odinger operators, etc.
\end{abstract}

\maketitle

%%%%%%%%%%%%%%%%%%%%%%%%%%%%%%%%%%%%%%
%%%%%%%%%%%%%%%% Section 1 %%%%%%%%%%%%%%%%
%%%%%%%%%%%%%%%%%%%%%%%%%%%%%%%%%%%%%%
\section{Introduction}

Let $\Omega$ be an open set in $\mathbb R^d$ with $d\ge1$. We consider
the Cauchy problem of energy-subcritical semilinear evolution equations:
\begin{equation}
\label{eq.NLH-sub}
\begin{cases}
	\partial_t u + L u +u= |u|^{p-1}u \quad &\text{in }(0,T)\times\Omega,\\
	u(0) = u_0 \in H^1(L)
\end{cases}
\end{equation}
with $1<p<p^*$,
and energy-critical semilinear evolution equations:
\begin{equation}
\label{eq.NLH-cri}
\begin{cases}
	\partial_t u + L u =  |u|^{\frac{4}{d-2}}u \quad &\text{in } (0,T)\times\Omega,\\
	u(0) = u_0 \in \dot H^1(L)
\end{cases}
\end{equation}
for $d\ge3$, where $T>0$, 
$u_0=u_0(x)$ is a given complex-valued function on $\Omega$,
$u=u(t,x)$ is an unknown complex-valued function on $[0,T)\times\Omega$,
$L$ is a self-adjoint operator on $L^2(\Omega)$ satisfying Assumption A or B below, 
and $p^*$ is the Sobolev 
critical exponent given by
\[
p^* =
\begin{cases}
\frac{d+2}{d-2}\quad &\text{if }d\ge3,\\
\infty &\text{if }d=1,2.
\end{cases}
\]
Here
$H^1(L)$ and $\dot H^1(L)$ are Sobolev spaces associated with $L$, and their norms are given by 
\[
\|f \|_{H^1(L)} := \| (I+L)^{\frac12}f\|_{L^2(\Omega)}
\quad \text{and}\quad
\|f \|_{\dot H^1(L)} := \| L^{\frac12}f\|_{L^2(\Omega)},
\]
respectively, where $I$ is the identity operator on $L^2(\Omega)$.
For precise definitions of $H^1(L)$ and $\dot H^1(L)$, 
we refer to Definition \ref{def:Sobolev} below.
For the sake of convenience we set $\mathcal E(L) = H^1(L)$ or $\dot H^1(L)$, and choose $\mathcal E(L) = H^1(L)$ in the case \eqref{eq.NLH-sub}
and $\mathcal E(L) = \dot H^1(L)$ in the case \eqref{eq.NLH-cri}.
The space $\mathcal E(L)$ is called the energy space associated with $L$. 
The energy functional $E_L : \mathcal E(L) \to \mathbb R$ is defined by
\begin{equation*}\label{eq.energy}
E_L (u)= \frac12 \|u\|_{\mathcal E(L)}^2 - \frac{1}{p+1} \|u\|_{L^{p+1}(\Omega)}^{p+1},
\end{equation*}
and the energy is formally dissipated along solutions to \eqref{eq.NLH-sub} and \eqref{eq.NLH-cri}:
\begin{equation}\label{eq.energy_diss}
\frac{d}{dt}E_L (u(t)) = -\int_\Omega |u_t(t)|^2\, dx \le 0.
\end{equation}

The problems \eqref{eq.NLH-sub} and \eqref{eq.NLH-cri} correspond to the energy-subcritical and critical cases in the following sense, respectively. 
The equation \eqref{eq.NLH-cri} with $L=-\Delta$ on $\mathbb R^d$, i.e., 
\begin{equation}\label{eq.heat}
\partial_t u - \Delta u = |u|^{p-1}u \quad\text{in } \mathbb R_+\times\mathbb R^d
\end{equation}
is invariant under the scale transformation
\[
u(t,x) \mapsto 
u_\lambda (t,x):=\lambda^{\frac{2}{p-1}}u(\lambda^2 t, \lambda x),\quad \lambda>0.
\]
Then 
\[
\|u_\lambda (0,\cdot)\|_{\dot H^{1}(\mathbb R^d)}
=
\lambda^{\frac{2}{p-1}-\frac{d-2}{2}}
\|u (0,\cdot)\|_{\dot H^{1}(\mathbb R^d)},
\]
and hence, if $p$ satisfies 
\[
\frac{2}{p-1}-\frac{d-2}{2} = 0 
\iff 
p=p^{*},
\]
then the $\dot H^{1}$-norm of initial data is invariant. 
Similarly, the energy is also invariant. Hence 
the case $p=p^*$ is called the energy-critical case, 
and the case $p<p^*$ (resp. $p>p^*$) is called the energy-subcritical case 
(resp. the energy-supercritical case).
Based on the above, we call the problems \eqref{eq.NLH-sub} and \eqref{eq.NLH-cri}  {\it energy-subcritical} and {\it energy-critical} in this paper, respectively.

The nonlinearity term $+|u|^{p-1}u$
of \eqref{eq.NLH-sub} and \eqref{eq.NLH-cri} is a sourcing term,
while the nonlinearity terms $-|u|^{p-1} u$ works as an absorbing term.
In the absorbing case,
all solutions exist globally in time and are dissipative, i.e., 
\[
\lim_{t\to\infty} \|u(t)\|_{\mathcal E(L)}=0,
\]
at least in the energy-subcritical case (see Remark \ref{rem:defo} below).
On the other hand, 
the behavior of solutions the equations with a sourcing term is completely different.
In this case,
the global behavior of solutions depends on initial data, that is,
the solutions are global (dissipation, asymptotical attraction by the ground state solution 
up to the scaling and translation, and blowing up in infinite time, etc.) or blow up in finite time.
Our purpose is to determine the global behavior of solutions by initial data $u_0$ 
with energy below the mountain pass level
\begin{equation}\label{eq.l_L}
l_L := \inf_{u\in \mathcal E(L)\setminus \{0\}} \max_{\lambda\ge0} E_L(\lambda u).
\end{equation}
For this purpose, let us introduce the Nehari functional and Nehari manifold:
\begin{equation*}\label{eq.J_L}
J_L (\phi) :=
\frac{d}{d\lambda} E_L (\lambda \phi)\big|_{\lambda=1}
=
\|\phi\|_{\mathcal E(L)}^2-\| \phi \|_{L^{p+1}(\Omega)}^{p+1},
\end{equation*}
\begin{equation*}\label{eq.Nehari}
\mathcal N_L :=
\big\{
\phi \in \mathcal E(L)\setminus \{0\} : J_L (\phi) = 0
\big\}.
\end{equation*}
Then the functional $J_L$ is formally written as
\begin{equation}\label{eq.J_L2}
J_{L}(u(t)) = -\frac12 \frac{d}{dt} \|u(t)\|_{L^2(\Omega)}^2
\end{equation}
for solutions $u$ to \eqref{eq.NLH-sub} or \eqref{eq.NLH-cri}.

In the case when $L=-\Delta$ on $\mathbb R^d$ or a bounded domain $\Omega$, 
the global behavior of solutions to \eqref{eq.NLH-sub} and \eqref{eq.NLH-cri} has been investigated. 
In particular, 
there are many literatures on their dynamics with energy below 
the mountain pass level, i.e., $E_{-\Delta}(u_0) < l_{-\Delta}$.  
In the energy-subcritical case, it was proved that 
the solution is global and dissipative 
if initial data belongs to the so-called stable set, while the solution blows up in finite time if 
initial data belongs to the so-called unstable set. 
Then the Nehari manifold 
plays an important role as a borderline separating 
the stable set and unstable set (see, e.g., \cite{GW-2005},\cite{IS-1996},\cite{I-1977},\cite{O-1981},\cite{PS-1975},\cite{T-1972}). 
In terms of the energy-critical case,
the pioneer works by Kenig and Merle 
are well known for focusing semilinear Schr\"odinger equations and wave equations on $\mathbb R^d$ with $d=3,4,5$ 
(see \cite{KM-2006}, \cite{KM-2008} and also Killip and Visan \cite{KV-2010} for Schr\"odinger equations with $d\ge6$).
Recently, Gustafson and Roxanas proved a similar result
for the semilinear heat equation \eqref{eq.heat} with $p=p^*$
for $d=4$ (see \cite{GR-2018}).

In the case when $E_{-\Delta}(u_0)= l_{-\Delta}$ and $J_{-\Delta}(u_0)= 0$, 
the solutions to \eqref{eq.NLH-sub} and \eqref{eq.NLH-cri} are global and stationary (not dissipative), because 
these problems have ground state solutions 
(see, e.g., \cite{BC-1987}, \cite{KM-2006}, \cite{T-1976}, \cite{T-1972}). 
On the other hand, if $E_{-\Delta}(u_0)= l_{-\Delta}$ and $J_{-\Delta}(u_0)\not= 0$, then 
the problem is reduced into the case $E_{-\Delta}(u_0)< l_{-\Delta}$, i.e., 
the solution is dissipative or blows up in finite.

The behavior of solutions to semilinear heat equations with energy above %or equal to 
the mountain pass level, i.e., $E_{-\Delta}(u_0)> l_{-\Delta}$, is completely different from the low energy case. 
In the energy-subcritical case, it was proved by Dickstein, Mizoguchi, Souplet and Weissler  
that the Nehari manifold is no longer the borderline (see \cite{DMSW-2011} and also Gazzola and  Weth \cite{GW-2005}). 
In the energy-critical case, 
Collot, Merle and Rapha\"el gave a classification of flow near the ground state solution for $d\ge7$.
More precisely, they proved that
one of the following three phenomenon always occurs: Global existence and asymptotical attraction by the ground state solution 
up to the scaling and translation; global existence and dissipation;
type I blow up (see \cite{CMR-2017} and references therein).
In the high energy case, Schweyer constructed type II blow up solutions for $d=4$.  
More precisely, 
for any $\varepsilon>0$, there exists a radially symmetric initial data $u_0 \in H^1(\mathbb R^4)$
with $l_{-\Delta}< E_{-\Delta}(u_0) < l_{-\Delta} + \varepsilon$ such that the solution to \eqref{eq.NLH-cri}
blows up in type II (see \cite{Sch-2012}).\\

In this paper we generalize the above results in the low energy case 
to more general self-adjoint operators $L$ with the following assumptions. 
In the subcritical case \eqref{eq.NLH-sub}, we assume the following $L^2$-$L^q$-estimates: \\

% Assumption A
\noindent{\bf Assumption A.} $L$ is a self-adjoint operator on $L^2(\Omega)$ such that
$\{e^{-tL}\}_{t>0}$ satisfies the following: For any $2 \le q < p^*+1$, there exist two constants $C>0$ and $0\le \omega<1$ such that
\begin{equation}
\label{eq.L2Lq}
\|e^{-tL}\|_{L^2(\Omega) \to L^q(\Omega)} \le C t^{-\frac{d}{2}(\frac{1}{2}-\frac{1}{q})} e^{\omega t}
\end{equation}
for any $t>0$.\\

In the critical case \eqref{eq.NLH-cri}, we assume the following Gaussian upper estimate:\\

% Assumption B
\noindent{\bf Assumption B.}
$L$ is a non-negative and self-adjoint operator on $L^2(\Omega)$
such that
the kernel $K_L(t;x,y)$ of $\{e^{-tL}\}_{t>0}$ satisfies the following: There exist two constants $c>0$ and $C>0$ such that
\begin{equation}
\label{eq.Gauss}
|K_L(t;x,y)| \le C t^{-\frac{d}{2}} e^{-\frac{|x-y|^2}{ct}}
\end{equation}
for any $t>0$ and almost everywhere $x,y\in\Omega$.\\

Note that Assumption B is stronger than Assumption A. 
Assumptions A and B are closely related with the Sobolev embeddings 
$H^1(L) \hookrightarrow L^{p+1}(\Omega)$ and $\dot H^1(L) \hookrightarrow L^{p^*+1}(\Omega)$, respectively, 
which play a fundamental role in well-definedness of energy and 
proving well-posedness for 
\eqref{eq.NLH-sub} and \eqref{eq.NLH-cri}, etc. 
In the following, let us define the inhomogeneous Sobolev spaces $H^s(L)$ of order $s \in \mathbb R$ under Assumption A and
the homogeneous ones $\dot H^s(L)$ of order $s$ under Assumption B, 
and state the Sobolev embedding theorem for $H^1(L)$ and $\dot H^1(L)$.

\begin{definition}\label{def:Sobolev}
\begin{itemize}
\item[(i)] Suppose that $L$ satisfies Assumption {\rm A}. Then for $s\in \mathbb R$ the inhomogeneous Sobolev space $H^s(L)$ is defined by
\[
H^s(L):=
\{ f \in \mathcal X'(L) : \|f\|_{H^s(L)}<\infty \}
\]
with the norm
\[
\| f \|_{H^s(L)}:=\|(I+ L)^{\frac{s}{2}} f\|_{L^2(\Omega)}.
\]
Here $\mathcal X'(L)$ is the topological dual of $\mathcal X(L)$ defined by
\[
\mathcal X(L)
:= \left\{
f \in L^1(\Omega) \cap \mathcal D(L): L^M f\in L^1(\Omega) \cap \mathcal D(L) \text{ for all }M\in\mathbb N
\right\},
\]
and $\mathcal D(L)$ denotes the domain of $L$.

\item[(ii)] Suppose that $L$ satisfies Assumption {\rm B}. Then for $s\in \mathbb R$ the homogeneous Sobolev space $\dot H^s(L)$ is defined by
\[
\dot H^s(L):=
\{ f \in \mathcal Z'(L) : \|f\|_{\dot H^s(L)}<\infty \}
\]
with the norm
\[
\| f \|_{\dot H^s(L)}:=\|L^{\frac{s}{2}} f\|_{L^2(\Omega)}.
\]
Here $\mathcal Z'(L)$ is the topological dual of $\mathcal Z(L)$ defined by
\[
\mathcal Z(L)
:= \left\{
f \in L^1(\Omega) \cap \mathcal D(L): L^M f\in L^1(\Omega) \cap \mathcal D(L) \text{ for all }M\in\mathbb Z
\right\}.
\]
\end{itemize}
\end{definition}

Then $H^s(L)$ and $\dot H^s(L)$ are well defined and complete.
We note that if $s\ge0$, then
\[
H^s(L)\cong
\{ f \in L^2(\Omega) : \|f\|_{H^s(L)}<\infty \},
\]
and the spaces $H^{-s}(L)$ and $\dot H^{-s}(L)$ are isomorphic to the adjoint spaces of $H^s(L)$ and $\dot H^s(L)$, respectively (see \cite{IMT-Besov} and Appendix~\ref{App:A}).
For these spaces, we have the Sobolev inequalities.

\begin{prop}\label{prop:Sobolev}
\begin{itemize}
\item[(i)] Suppose that $L$ satisfies Assumption~{\rm A}. Then
for any $1 < p < p^*$, there exists a constant $C>0$ such that
\begin{equation}\label{eq.Sobolev1}
\| f \|_{L^{p+1}(\Omega)}
\le C
\| f \|_{H^1(L)}
\end{equation}
for any $f \in H^1(L)$.
\item[(ii)] Let $d\ge3$. Suppose that $L$ satisfies Assumption {\rm B}. Then
there exists a constant $C>0$ such that
\begin{equation}\label{eq.Sobolev2}
\| f \|_{L^{\frac{2d}{d-2}}(\Omega)}
\le C
\| f \|_{\dot H^1(L)}
\end{equation}
for any $f \in \dot H^1(L)$.
\end{itemize}
\end{prop}

For the proof we refer to Appendix \ref{App:B}.

\begin{rem}
The number $l_{L}$ in \eqref{eq.l_L} is also characterized by the best constant of the Sobolev inequality. 
 More precisely, 
\begin{equation}\label{eq.l_L2}
l_L = \inf_{f \in \mathcal N_L} E_L(f) = \frac{p-1}{2(p +1)} S_{p+1}^{-\frac{2(p +1)}{p-1}},
\end{equation}
where $S_{p+1}=S_{p+1}(d,L)$ is the best constant of the Sobolev inequality \eqref{eq.Sobolev1} if $1<p<p^*$,
and \eqref{eq.Sobolev2} if $d\ge3$ and $p=p^*$
(see Appendix \ref{App:C}).
\end{rem}

A typical example of $L$ is the Laplace operator $-\Delta$ on $L^2(\mathbb R^d)$. In the rest of this section, let us give other major examples of $L$ satisfying Assumption~A or B.

\begin{itemize}

\item[(a)] ({\bf The Schr\"odinger operator with the Dirichlet boundary condition})
The Schr\"odinger operator $-\Delta_D + V$ with the Dirichlet boundary condition on an open set $\Omega$ of $\mathbb R^d$ with $d\ge1$
satisfies {\it Assumption {\rm A}}, where $V=V(x)$ is a real-valued measurable function on $\Omega$ such that
the infimum of the spectrum of $-\Delta_D + V$ is strictly larger than $-1$,
and
\[
V=V_+ - V_-,\quad V_{\pm} \ge0,\quad V_+ \in L^1_{\mathrm{loc}}(\Omega)\quad \text{and}\quad
V_- \in K_d(\Omega)
\]
(see, e.g., Propositions 2.1 and 3.1 in \cite{IMT-RMI}).
We say that $V_-$ belongs to the Kato class $K_{d}(\Omega)$ if
\begin{align}\notag
\left\{
\begin{aligned}
	&\lim_{r \rightarrow 0} \sup_{x \in \Omega} \int_{\Omega \cap \{|x-y|<r\}}
	   \frac{V_-(y)}{|x-y|^{d-2}} \,dy = 0 &\text{for }d\ge 3, \\
	&\lim_{r \rightarrow 0} \sup_{x \in \Omega} \int_{\Omega \cap \{|x-y|<r\}}
	   \log (|x-y|^{-1})V_-(y) \,dy = 0 &\text{for }d=2, \\
	&\sup_{x \in \Omega}\int_{\Omega \cap \{|x-y|<1\}} V_-(y) \,dy <\infty &\text{for }d=1\,
	\end{aligned}\right.
\end{align}
(see Section A.2 in Simon \cite{Simon-1982}).
It is readily seen that the potential $V_-(x) = 1/|x|^\alpha$ with $0\le \alpha <2$ if $d\ge2$ and $0\le \alpha <1$ if $d=1$
is included in $K_{d}(\Omega)$. 
%(see Fukaya and Ohta \cite{FO-appear}).
It should be noted that the potential like $V_-(x) = 1 / |x|^2$ as $|x|\to 0$
is excluded from $K_d(\Omega)$ 
(see Example (e) below).

In addition, if the negative part $V_-$ satisfies
\[
\begin{cases}
\displaystyle\sup_{x \in \Omega} \int_{\Omega} \frac{V_-(y)}{|x-y|^{d-2}} \,dy
 < \dfrac{\pi^{\frac{d}{2}}}{\Gamma (d/2-1)}
& \quad \text{if } d \geq 3,\\
V _- = 0
& \quad \text{if } d = 1, 2,
\end{cases}
\]
then $-\Delta_D + V$ satisfies {\it Assumption {\rm B}} (see Propositions 2.1 and 3.1 in \cite{IMT-RMI}).
In particular, $-\Delta_D$ (i.e., the case of $V=0$) satisfies {\it Assumption {\rm B}}.

\item[(b)] ({\bf The Neumann Laplacian}) Let $\Omega$ be a domain of $\mathbb R^d$ having the extension property (see, e.g., Davies \cite{D_1989}).
Then the Laplace operator $-\Delta_N$ with the Neumann boundary condition on $\Omega$
satisfies {\it Assumption {\rm A}}.
Indeed, when $\Omega$ has the extension property, the following Sobolev inequality holds:
\[
\| f \|_{L^{p^*+1}(\Omega)}
\le C
\| f \|_{H^1(\Omega)}
\]
for any $f \in H^1(\Omega)$.
Then, applying the above estimate to $f=e^{t\Delta_N}u_0$, we have 
\[
\|e^{t\Delta_N}\|_{L^2(\Omega)\to L^{p^*+1}(\Omega)}
\le C (t^{-\frac12} + 1)
\]
for any $t>0$. Hence 
we obtain the estimate \eqref{eq.L2Lq} with $\omega=0$ 
by combining 
the Riesz-thorin interpolation theorem with the above estimate and $L^{2}$-boundedness of $e^{t\Delta_N}$. Thus $-\Delta_{N}$ satisfies  {\it Assumption {\rm A}}. 
However $-\Delta_{N}$ does not satisfies {\it Assumption {\rm B}} in general. 
Indeed, if $\Omega$ is bounded, then 
$-\Delta_{N}$ has the zero eigenvalue.
This implies that 
$e^{t\Delta_{N}}$ does not satisfy the Gaussian upper bound \eqref{eq.Gauss} for $t>1$ in {\it Assumption {\rm B}}.

\item[(c)] ({\bf The Robin Laplacian on an exterior domain})
Let $d\ge3$ and $\Omega$ be the exterior domain in $\mathbb R^d$ of a compact and connected set with Lipschitz boundary.
We consider the Laplace operator $-\Delta_\sigma$ on $L^2(\Omega)$ associated with a quadratic form
\[
q_\sigma(f,g) = \int_\Omega \nabla f \cdot \overline{\nabla g}\,dx
+ \int_{\partial \Omega} \sigma f \overline{g}\,dS
\]
for any $f,g \in H^1(\Omega)$, where $\sigma$ is a function $\partial \Omega \to \mathbb R$ and $\partial \Omega$ denotes the boundary of $\Omega$.
Note that $-\Delta_0$ (i.e., the case of $\sigma=0$) is the Neumann Laplacian on $L^2(\Omega)$.
Assume that $\sigma \in L^\infty(\partial \Omega)$ and $\sigma\ge0$.
Then $-\Delta_\sigma$ satisfies {\it Assumption~{\rm B}}.
This is a consequence of the following two estimates:
\begin{equation}\label{eq.Robin1}
0 \le K_{-\Delta_\sigma}(t; x,y) \le K_{-\Delta_0}(t; x,y), 
\end{equation}
\begin{equation}\label{eq.Robin2}
0\le K_{-\Delta_0}(t;x,y)
\le C t^{-\frac{d}{2}} e^{-\frac{|x-y|^2}{ct}}
\end{equation}
for any $t>0$ and almost everywhere $x,y\in\Omega$. 
The estimate \eqref{eq.Robin1} follows from domination of semigroups, 
and the proof of \eqref{eq.Robin2} can be found in 
Chen, Williams and Zhao \cite{CWZ-1994}.

In the case $d=2$, if $\Omega$ is the exterior domain in $\mathbb R^2$ of a compact and connected set with $C^2$-boundary and
\[
\underset{x\in\partial \Omega}{\mathrm{ess\ inf}}\,
\sigma(x) >0,
\]
then $-\Delta_\sigma$ satisfies {\it Assumption~{\rm B}}
(see Section 2 in Kova\v{r}\'{i}k and Mugnolo \cite{KM-2018}).

In the case $d=1$, let $\Omega=\mathbb R_+$ and $-\Delta_\sigma$
is the Laplace operator on $L^2(\mathbb R_+)$ associated with a quadratic form
\[
q_\sigma(f,g) = \int_0^\infty f' \overline{g'}\,dx
+ \sigma f(0)\overline{g(0)}
\]
for any $f,g \in H^1(\mathbb R_+)$, where $\sigma\ge 0$ is a constant.
Then $-\Delta_\sigma$ satisfies {\it Assumption~{\rm B}}
(see Section 4 in \cite{KM-2018}).

\item[(d)] ({\bf The elliptic operator})
Let $L$ be the self-adjoint operator associated with a quadratic form
\[
q(f,g) =
\int_{\mathbb R^d}
\left\{
\sum_{k,j=1}^d a_{kj} D_k f \overline{D_j g} + \sum_{k=1}^d
(b_k \overline{g} D_k f + c_k f \overline{D_k g}) + a_0 f \overline{g}
\right\}\,dx
\]
for any $f, g \in H^1(\mathbb R^d)$, where
$a_{kj}, b_k, c_k, a_0 \in L^\infty(\mathbb R^d)$ are real-valued functions for all $1\le j,k\le d$, and the principle part is elliptic, i.e.,
there exists a constant $\eta>0$ such that
\[
\sum_{j,k=1}^d a_{kj}(x) \xi_j \overline{\xi_k} \ge \eta |\xi|^2,\quad \xi \in\mathbb C^{d}, \text{ a.e.}\, x \in \mathbb R^d.
\]
Then $L$ is self-adjoint on $L^2(\mathbb R^d)$ and satisfies the Gaussian upper estimate: There exist three constants $c>0$, $C>0$ and $\omega\ge0$ such that
\[
|K_L(t;x,y)| \le C t^{-\frac{d}{2}} e^{\omega t} e^{-\frac{|x-y|^2}{ct}},\quad t>0,\quad  \text{a.e.}\, x,y\in\mathbb R^{d}
\]
(see, e.g., \cite{O_2005}).
If $\omega< 1$, then $L$ satisfies {\it Assumption {\rm A}}.

\item[(e)] ({\bf The Schr\"odinger operator with a negative inverse-square potential})
The Schr\"odinger operator $- \Delta - c/|x|^2$
on $\mathbb R^d$ with $d\ge3$, where
\[
0< c \le \frac{(d-2)^2}{4},
\]
satisfies {\it Assumption {\rm A}}, and not {\it Assumption {\rm B}} (see \cite{IMSS-2016}, \cite{IO-appear}).

\end{itemize}

\medskip

The rest of this paper is organized as follows.
In Section \ref{sec:2} we state main results on the global behavior of solutions to \eqref{eq.NLH-sub} or \eqref{eq.NLH-cri}.
In Section \ref{sec:3} we provide the results on local well-posedness of \eqref{eq.NLH-sub} and \eqref{eq.NLH-cri}, respectively.
In Section \ref{sec:4} we show some lemmas on variational estimates.
In Section \ref{sec:5} the proofs of main results are given.

%%%%%%%%%%%%%%%%%%%%%%%%%%%%%%%%%%%%%%
%%%%%%%%%%%%%%%% Section 2 %%%%%%%%%%%%%%%%
%%%%%%%%%%%%%%%%%%%%%%%%%%%%%%%%%%%%%%
\section{Statements of main results}\label{sec:2}

\subsection{The subcritical case}
First let us give a definition of (energy finite) solution to \eqref{eq.NLH-sub}. 

\begin{definition}\label{def:1}
Let $T \in (0,\infty]$.
A function $u : [0,T) \times \Omega \rightarrow \mathbb C$ is called a solution
to \eqref{eq.NLH-sub} on $[0,T) \times \Omega$ if
$u \in C([0,T']; H^1(L))$ and $u_t, Lu \in L^2([0,T']\times\Omega)$
for any $T' \in [0,T)$,
and it satisfies the Duhamel formula
\begin{equation}\label{eq.Duhamel1}
u(t) = e^{-t(I+L)}u_0 + \int_0^t e^{-(t-s)(I+L)} |u(s)|^{p-1}u(s) \,ds
\end{equation}
for any $t\in [0,T)$. The time $T$ is said to be the maximal existence time if the solution cannot be extended beyond $[0,T)$, and
we denote by $T_m$ the maximal existence time 
(This is well-defined by the uniqueness (ii) in Proposition \ref{thm:LW-sub} below).
We say that $u$ is a global solution if $T_m=+\infty$, and that $u$ blows up in finite time if $T_m<+\infty$.
\end{definition}

Our main result in the subcritical case is the following theorem, which states the dichotomy of dissipation and blowing up in finite time under the assumption $E_L (u_0) \le l_L$.

\begin{thm}\label{thm:GB-sub}
Suppose that $L$ satisfies Assumption {\rm A}.
Let $u$ be a solution to the problem \eqref{eq.NLH-sub} with initial data $u_0\in H^1(L)$
satisfying $E_L (u_0) \le l_L$.
Then the following assertions hold{\rm :}
\begin{itemize}
\item[(i)] If $J_L (u_0)> 0$, then $T_m = +\infty$ and
\[
\|u(t)\|_{H^1(L)}=o(t^{-\frac12})\quad \text{as $t\to \infty$}. 
\]
\item[(ii)] If $J_L (u_0) < 0$, then $T_m < +\infty$.
\end{itemize}
\end{thm}

\begin{figure}[htbp]
\begin{center}
\includegraphics[width=140mm]{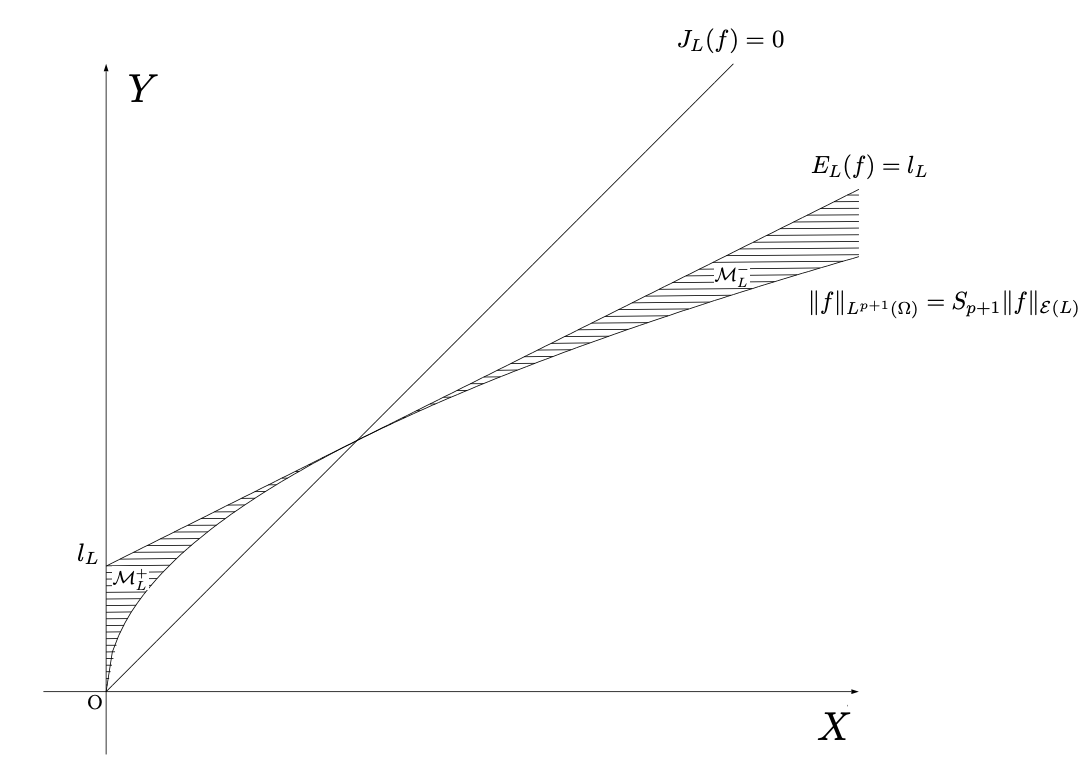}
\caption{$X$-axis: $X=\frac{1}{p+1}\|f\|_{L^{p+1}}^{p+1}$; $Y$-axis: $Y=\frac12 \|f \|_{\mathcal E(L)}^{2}$}\label{fig:1}
\end{center}
\end{figure}

Let us define
\[
\mathcal M^{+}_L :=
\{
f \in \mathcal E(L) : E_L (f) \le l_L,\ J_L (f) > 0
\}\cup \{0\},
\]
\[
\mathcal M^{-}_L :=
\{
f \in \mathcal E(L) : E_L (f) \le l_L,\ J_L (f) < 0
\}.
\]
Theorem \ref{thm:GB-sub} states 
the solution is global and dissipative 
if $u_{0}$ belongs to the so-called stable set $\mathcal M^{+}_L$, while the solution blows up in finite time if 
$u_{0}$ belongs to the so-called unstable set $\mathcal M^{-}_L$.
Then the Nehari manifold $\mathcal N_{L}$
is a borderline separating $\mathcal M^{+}_L$ and $\mathcal M^{-}_L$ (see Figure \ref{fig:1}).

\begin{rem}\label{rem:equal}
Let us give three remarks on the case when $J_L (u_0) = 0$.
\begin{itemize}
\item[(i)]
There is no function $u_0 \in H^1(L)\setminus \{0\}$ such that $E_L (u_0) < l_L$ and $J_L (u_0) = 0$,
because the existence of such a function contradicts \eqref{eq.l_L2}.

\item[(ii)] Let $u_0 \in H^1(L)$ with $E_L (u_0) = l_L$ and $J_L (u_0) = 0$.
Then the solution $u$ must be one of the following three solutions:
A stationary solution;
a global (in time) solution decaying to zero as $t \to \infty$;
a blow-up solution in finite time.
In fact,
if 
\begin{equation}\label{eq.ec}
\text{$E_L(u(t)) = E_L (u_0)$ for any $t \in [0,T_m)$,}
\end{equation}
then the solution $u$ must be a ground state solution (i.e., a minimal energy nontrivial solution) to the stationary problem
\begin{equation}\label{eq.stationary}
\begin{cases}
Lv + v = |v|^{p-1}v\quad \text{in } \Omega,\\
v \in H^1(L).
\end{cases}
\end{equation}
Therefore, if \eqref{eq.stationary} has a ground state solution, then $u$ is stationary.
If \eqref{eq.stationary} has no ground state solution,
then \eqref{eq.ec} does not occur.
When \eqref{eq.ec} does not hold, we see from Theorem \ref{thm:GB-sub} that
$u$ is dissipative or blows up in finite.
\end{itemize}
\end{rem}

\begin{rem}\label{rem:defo}
In the absorbing and energy-subcritical case,
all solutions are global and dissipative.
Indeed, in this case, we have
\[
\|u(t)\|_{H^1(L)}^2 \le 2E_L(u(t)) \le 2E_L(u_0)\quad \text{for any $t \in [0,T_m)$},
\]
which shows $T_m=+\infty$ by (iv) in Proposition \ref{thm:LW-sub} below.
Furthermore, it follows from \eqref{eq.J_L2} that
\[
\|u(t)\|_{L^2(\Omega)}^2
+
2 \int_0^t
J_L (u(s))\,ds
=
\|u_0\|_{L^2(\Omega)}^2
\]
for any $t>0$. Since $J_L(u(t)) \ge \|u(t)\|_{H^1(L)}^2$ for any $t>0$,
we find that
\[
2\|u\|_{L^2(\mathbb R_+; H^1(L))}^2
\le \sup_{t\ge0} \left(\|u(t)\|_{L^2(\Omega)}^2 +2\|u\|_{L^2((0,t); H^1(L))}^2\right) \le \|u_0\|_{L^2(\Omega)}^2,
\]
which proves that $u(t) \to 0$ in $H^1(L)$ as $t\to\infty$.
\end{rem}

\begin{rem}
If the infimum of spectrum of $L$ is strictly positive, then we have
\[
\| f \|_{L^{p+1}(\Omega)}
\le C
\| f \|_{\dot H^1(L)}
\]
for any $f \in H^1(L)$.
Hence, by using the above inequality instead of \eqref{eq.Sobolev1}, 
we can prove the same statements as in Theorem~\ref{thm:GB-sub} for 
 the problem
\[
\begin{cases}
	\partial_t u + L u= |u|^{p-1}u \quad &\text{in }(0,T)\times\Omega,\\
	u(0) = u_0 \in H^1(L)
\end{cases}
\]
with $1<p<p^*$.
\end{rem}

\subsection{The critical case}
Let us give a definition of (energy finite) solution to \eqref{eq.NLH-cri}. 

\begin{definition}\label{def:2}
Let $T \in (0,\infty]$.
A function $u : [0,T) \times \Omega \rightarrow \mathbb R$ is called a solution
to \eqref{eq.NLH-cri} if 
$u \in C([0,T']; \dot H^1(L)) \cap L^{\frac{2(d+2)}{d-2}} ((0,T']\times\Omega)$
and 
$u_t, Lu \in L^2([0,T']\times\Omega)$
for any $T' \in [0,T)$,
and it satisfies the Duhamel formula
\begin{equation}\label{eq.Duhamel2}
u(t) = e^{-tL}u_0 + \int_0^t e^{-(t-s)L} |u(s)|^{p-1}u(s) \,ds
\end{equation}
for any $t\in [0,T)$.
The definitions of the maximal existence time $T_m$, global solution, and 
blow up solution in finite time are the same as in Definition \ref{def:1}.
\end{definition}

In the critical case, we also have the similar result to Theorem \ref{thm:GB-sub}, 
i.e., the dichotomy of dissipation and blowing up in finite time under the assumption $E_L (u_0) \le l_L$.

\begin{thm}\label{thm:GB-cri}
Suppose that $L$ satisfies Assumption {\rm B}.
Let $u$ be a solution to the problem \eqref{eq.NLH-cri} with initial data $u_0\in \dot{H}^1(L)$
satisfying $E_L (u_0) \le l_L$.
Then the following assertions hold{\rm :}
\begin{itemize}
\item[(i)] If $J_L (u_0)> 0$ and 
$\|e^{-tL}u_0\|_{S(\mathbb R_+)}$ is sufficiently small, then $T_m = +\infty$ and
\[
\lim_{t\to\infty}\|u(t)\|_{\dot{H}^1(L)}=0,
\]
where $\|\cdot\|_{S(\mathbb R_+)}$ is the space-time norm given by \eqref{eq.St-norm} below.
\item[(ii)] If $J_L (u_0) < 0$, then $T_m < +\infty$.
\end{itemize}
\end{thm}

Compared with the subcritical case, 
we impose the assumption on smallness of initial data 
in the assertion (i) of Theorem~\ref{thm:GB-cri}. 
This comes from the fact that the maximal existence time $T_m$ depends on 
not only the size of $u_0$, but also 
its profile.
In the case when $L=-\Delta$ on $L^2(\mathbb R^d)$ with $d=4$, 
it is known that 
the assumption can be removed via the linear profile decomposition
plus backward uniqueness (see \cite{GR-2018}).

In the case when $L=-\Delta$ on $L^2(\mathbb R^d)$ with $d=4$, 
the statement (ii) with the additional assumption $u_{0}\in L^{2}(\mathbb R^{d})$ is proved by \cite{GR-2018}. In this paper 
we succeed in removing the additional assumption by combining the blow up argument in \cite{GR-2018} with a cut-off argument.

\begin{rem}\label{rem:equal2}
The similar statements to Remark \ref{rem:equal} hold in the energy-critical case.
\end{rem}

\begin{rem}
In studying Dirichlet problems of partial differential equations, 
the following homogeneous Sobolev space $\dot H^{1}_{0}(L)$ as the completion of $C^\infty_0(\Omega)$ with respect to $\|\cdot\|_{\dot H^1(L)}$ is often used:
\[
\dot H^1_0(L) := \overline{C^\infty_0(\Omega)}^{\|\cdot\|_{\dot H^1(L)}}
\]
(see, e.g., \cite{KVZ-2016b}). 
The space $\dot H^1_0(L)$ does not coincide with $\dot H^1(L)$ in general, 
but when $L=-\Delta$ on $L^2(\mathbb R^d)$, we have
\[
\dot H^1_0(-\Delta) = \dot H^1(-\Delta) \left(= \dot H^1(\mathbb R^d)\right).
\]
When we consider initial data $u_{0} \in \dot H^1_0(L)$, 
we can prove the same statements as in Theorem \ref{thm:GB-cri}
under the following weaker assumption than Assumption B: 
For any $2\le q\le r\le \infty$ there exists a constant $C>0$ such that
\begin{equation}\label{eq.assC}
\|e^{-tL}\|_{L^q(\Omega) \to L^r(\Omega)} \le C t^{-\frac{d}{2}(\frac{1}{q}-\frac{1}{r})}\end{equation}
for any $t>0$. Indeed, 
the Sobolev inequality \eqref{eq.Sobolev2} for $f\in \dot H^1_0(L)$ 
is assured by \eqref{eq.assC} 
(see Theorem 6.4 in Chapter 6 from Ouhabaz~\cite{O_2005}).
By using this Sobolev inequality instead of (ii) in Proposition \ref{prop:Sobolev}, 
we can obtain the statements in Theorem~\ref{thm:GB-cri}
replacing $\dot H^{1}(L)$ with $\dot H^1_0(L)$. 
The proof is the same as 
the case of $\dot{H}^1(L)$. So we may omit the details.
\end{rem}

%%%%%%%%%%%%%%%%%%%%%%%%%%%%%%%%%%%%%%
%%%%%%%%%%%%%%%% Section 3 %%%%%%%%%%%%%%%%
%%%%%%%%%%%%%%%%%%%%%%%%%%%%%%%%%%%%%%
\section{Local Theory}\label{sec:3}

%%%%%%%%%%%%%%%% Subsection 3.1 %%%%%%%%%%%%%%%%
\subsection{The subcritical case}
In this subsection we state a result on local well-posedness for the problem \eqref{eq.NLH-sub}.
For this purpose, we prepare the following:

\begin{lem}\label{lem:L2Lq}
Suppose that $L$ satisfies Assumption {\rm A}. Then for any $2 \le q < p^*+1$,
there exists a constant $C>0$ such that
\begin{equation}\label{eq.L2Lq2}
\|e^{-t(I+L)}\|_{L^2(\Omega) \to L^q(\Omega)} + \|e^{-t(I+L)}\|_{L^{q'}(\Omega) \to L^2(\Omega)} \le C t^{-\frac{d}{2}(\frac{1}{2}-\frac{1}{q})}
\end{equation}
for any $t>0$, where $1/q + 1/q' = 1$.
\end{lem}
The estimate for the first term in the left hand side of \eqref{eq.L2Lq2} immediately follows from Assumption A,
and the estimate for the second term is obtained by the duality argument.\\

The local well-posedness for \eqref{eq.NLH-sub} is proved by
the fixed point argument, Lemma \ref{lem:L2Lq} and
the Sobolev inequality \eqref{eq.Sobolev1} in Proposition \ref{prop:Sobolev} (see, e.g., Cazenave and Weissler \cite{CW-1990}).
More precisely, we have the following:

\begin{prop}\label{thm:LW-sub}
Suppose that $L$ satisfies Assumption {\rm A}.
Let $u_0 \in H^1(L)$.
Then the following assertions hold{\rm :}

\begin{itemize}
\item[(i)] {\rm (Existence)} 
There exists a maximal existence time $T_m>0$, depending only on $\|u_0\|_{H^1(L)}$, such that
there exists a solution $u$ to \eqref{eq.NLH-sub} on $[0, T_m)\times \Omega$ with $u(0)=u_0$ in the sense of Definition \ref{def:1}.

\item[(ii)] {\rm (Uniqueness)} Let $T>0$. If $u_1, u_2 \in L^\infty([0,T]; H^1(L))$ 
satisfy the equation \eqref{eq.Duhamel1}
with $u_1(0)=u_2(0)=u_0$, then $u_1=u_2$ on $[0,T]$.

\item[(iii)] {\rm (Continuous dependence on initial data)}
The function $T_m : H^1(L) \to (0,\infty]$ is lower semicontinuous.
Furthermore, if $u_{0,n} \to u_0$ in $H^1(L)$ as $n\to\infty$ and $u_n$ is a solution to \eqref{eq.NLH-sub} with $u_n(0)=u_{0,n}$,
then $u_n \to u$ in $L^\infty([0,T];H^1(L))$ as $n\to\infty$ for any $0<T<T_m$.

\item[(iv)] {\rm (Blow-up criterion)}
If $T_m < +\infty$, then $\displaystyle \lim_{t\to T_m}\|u(t)\|_{H^1(L)} = +\infty$.

\item[(v)] {\rm (Energy identity)}
A solution $u$ to \eqref{eq.NLH-sub} in $[0, T_m)\times \Omega$ with $u(0)=u_0$ in the sense of Definition~\ref{def:1} satisfies the identity
\[
E_L(u(t)) + \int_0^t\int_\Omega |u_t|^2\, dx dt =  E_L(u_0),\quad 0<t<T_m.
\]
\end{itemize}
\end{prop}

%%%%%%%%%%%%%%%% Subsection 3.2 %%%%%%%%%%%%%%%%
\subsection{The critical case}
In this subsection we state a result on the case of the problem \eqref{eq.NLH-cri}.
For this purpose, we prepare the following:

\begin{lem}\label{lem:Str}
Suppose that $L$ satisfies Assumption {\rm B}.
Then the following assertions hold{\rm :}
\begin{itemize}
\item[(i)] For any $1\le q \le r \le \infty$, there exists a constant $C>0$ such that
\[
\|e^{-tL} \|_{L^q(\Omega) \to L^r(\Omega)} \le C t^{-\frac{d}{2}(\frac{1}{q}-\frac{1}{r})}
\]
for any $t>0$.

\item[(ii)] Let $1< q \le r \le \infty$ and
\[
\frac{1}{\gamma}= \frac{d}{2} \left(\frac{1}{q}-\frac{1}{r}\right).
\]
Then there exists a constant $C>0$ such that
\begin{equation}\label{eq.Str1}
\|e^{-tL} f\|_{L^\gamma (\mathbb R_+; L^{r}(\Omega))}
\le C
\| f \|_{L^{q}(\Omega)}
\end{equation}
for any $f \in L^{q}(\Omega)$.

\item[(iii)]
There exists a constant $C>0$ such that
\begin{equation}\label{eq.Str2}
\|e^{-tL} f\|_{L^2 (\mathbb R_+; \dot H^1(L))}
\le C
\| f \|_{L^2(\Omega)}
\end{equation}
for any $f \in L^2(\Omega)$.

\item[(iv)] Let $1\le q_1\le q_2\le \infty$ and $1 < \gamma_1, \gamma_2 < \infty$ satisfying
\[
\frac{1}{\gamma_2} = \frac{1}{\gamma_1} + \frac{d}{2} \left( \frac{1}{q_1} - \frac{1}{q_2}\right) -1,\quad
\frac{d}{2} \left( \frac{1}{q_1} - \frac{1}{q_2}\right) < 1.
\]
Then there exists a constant $C>0$ such that
\begin{equation}\label{eq.Str3}
\left\|
\int_0^t e^{-(t-s)L} F(s)\,ds
\right\|_{L^{\gamma_2} (\mathbb R_+; L^{q_2}(\Omega))}
\le C
\| F \|_{L^{\gamma_1} (\mathbb R_+; L^{q_1}(\Omega))}.
\end{equation}

\item[(v)] There exists a constant $C>0$ such that
\begin{equation}\label{eq.Str4}
\left\|
\int_0^t e^{-(t-s)L} F(s)\,ds
\right\|_{L^\infty (\mathbb R_+; \dot{H}^{1}(L))}
\le C
\| F \|_{L^{2} (\mathbb R_+; L^2(\Omega))}.
\end{equation}
\end{itemize}
\end{lem}

\begin{proof}
The assertion (i) is an immediate consequence of Assumption B.
The proof of (ii) is based on the method of Weissler \cite{Wei-1981} and Giga \cite{Gig-1986},
in which main tools are the assertion (i) and the Marcinkiewicz interpolation theorem
(cf. Miao, Yuan and Zhang \cite{MYZ-2008}).
The assertion (iv) can be proved by combining the assertion (i) and the Hardy-Littlewood-Sobolev inequality.
Finally, we obtain the assertions (iii) and (v)
in the same argument as in the proofs of Propositions 2.2 and 2.4 in \cite{WHHG_2011}, respectively.
So we may omit the details.
\end{proof}

We define $S(I):= L^{\frac{2(d+2)}{d-2}} (I\times \Omega)$ and 
the space-time norm by
\begin{equation}\label{eq.St-norm}
\|u\|_{S(I)} := \|u\|_{L^{\frac{2(d+2)}{d-2}} (I\times \Omega)}
=
\left( \int_I \int_\Omega |u(t,x)|^\frac{2(d+2)}{d-2}\, dx dt\right)^\frac{d-2}{2(d+2)}.
\end{equation}
for an interval $I\subset \mathbb R _+$.
Then we have the following result on local well-posedness for the problem \eqref{eq.NLH-cri}:

\begin{prop}\label{thm:LW-cri}
Suppose that $L$ satisfies Assumption {\rm B}.
Let $u_0 \in \dot H^1(L)$.
Then the following assertions hold{\rm :}
\begin{itemize}

\item[(i)] {\rm (Existence)}
There exists a maximal existence time $T_m>0$ such that
there exists a solution $u$ to \eqref{eq.NLH-cri} in $[0, T_m)\times \Omega$ with $u(0)=u_0$ in the sense of Definition~\ref{def:2}. 

\item[(ii)] {\rm (Uniqueness)} Let $T>0$. If $u_1, u_2 \in S ((0,T])$ 
satisfy the equation \eqref{eq.Duhamel2} with $u_1(0)=u_2(0)=u_0$, then $u_1=u_2$ on $[0,T]$.

\item[(iii)] {\rm (Continuous dependence on initial data)}
The function $T_m : \dot H^1(L) \to (0,\infty]$ is lower semicontinuous.
Furthermore, if $u_{0,n} \to u_0$ in $\dot H^1(L)$ as $n\to\infty$ and $u_n$ is a solution to \eqref{eq.NLH-cri} with $u_n(0)=u_{0,n}$,
then $u_n \to u$ in $L^\gamma([0,T];\dot H^1(L))$ as $n\to\infty$ for any $1\le \gamma < \infty$ and $0<T<T_m$.

\item[(iv)] {\rm (Blow-up criterion)} If $T_m < +\infty$, then $\| u\|_{S((0,T_m))} = +\infty$.

\item[(v)] {\rm (Energy identity)} 
A solution $u$ to \eqref{eq.NLH-cri} in $[0, T_m)\times \Omega$ with $u(0)=u_0$ in the sense of Definition~\ref{def:2} satisfies the identity
\begin{equation}\label{eq.energy-identity}
E_L(u(t)) + \int_0^t\int_\Omega |u_t|^2\, dx dt =  E_L(u_0),\quad 0<t<T_m.
\end{equation}

\item[(vi)] {\rm (Small data global existence)} There exists $\varepsilon_0>0$ such that
if $\|e^{-tL}u_0\|_{S(\mathbb R_+)} < \varepsilon_0$, then $T_m = +\infty$ and
\[
\| u \|_{S(\mathbb R_+)} \le 2\varepsilon_0.
\]
In particular, if $\|u_0\|_{\dot{H}^1(L)}$ is sufficiently small, then $\|e^{-tL}u_0\|_{S(\mathbb R_+)} < \varepsilon_0$.
\end{itemize}
\end{prop}

We note that $T_m$ depends on the profile of $u_0$, not only the size of $u_0$. 
We apply the fixed point argument and space-time estimates in Lemma \ref{lem:Str} 
to prove Proposition~\ref{thm:LW-cri}. 
For the details of proof, see Appendix \ref{App:D}.

%%%%%%%%%%%%%%%%%%%%%%%%%%%%%%%%%%%%%%
%%%%%%%%%%%%%%%% Section 4 %%%%%%%%%%%%%%%%
%%%%%%%%%%%%%%%%%%%%%%%%%%%%%%%%%%%%%%
\section{Variational estimates}\label{sec:4}

In this section we show some lemmas on the elementary variational inequalities.
We recall $\mathcal E(L) = H^1(L)$ or $\dot H^1(L)$, and choose $\mathcal E(L) = H^1(L)$ in the case \eqref{eq.NLH-sub}
and $\mathcal E(L) = \dot H^1(L)$ in the case \eqref{eq.NLH-cri}.
In this and next sections, we consider only the case $E_L (u_0) < l_L$, since 
the problem in another case $E_L (u_0) = l_L$ and $J_L(u_0)\not =0$
is reduced to the case $E_L (u_0) < l_L$ (see Remarks \ref{rem:equal} and \ref{rem:equal2}).\\

Let us define the stable and unstable sets 
\[
\mathcal M^{+}_L :=
\{
f \in \mathcal E(L) : E_L (f) < l_L,\ J_L (f) > 0
\}\cup\{0\},
\]
\[
\mathcal M^{-}_L :=
\{
f \in \mathcal E(L) : E_L (f) < l_L,\ J_L (f) < 0
\},
\]
respectively.
In the following, we denote by $u=u(t)$ the solution to \eqref{eq.NLH-sub} or \eqref{eq.NLH-cri} with initial data $u_0$.
The following lemma states that $\mathcal M^{\pm}_L$ are invariant under the semiflow associated to \eqref{eq.NLH-sub} or \eqref{eq.NLH-cri}.

\begin{lem}\label{lem:Mpm}
If $u_0 \in \mathcal M^{\pm}_L$,
then $u(t) \in \mathcal M^{\pm}_L$ for any $t \in [0,T_m)$,
where double-sign corresponds.
\end{lem}

\begin{proof}
Let $u_0 \in \mathcal M^{+}_L$.
Then $u(t)\in \mathcal M^{+}_L \cup \mathcal M^{-}_L$ for any $t \in [0,T_m)$,
since
\begin{equation}\label{eq.ei}
E_L (u(t)) \le E_L (u_0)
\end{equation}
for any $t \in [0,T_m)$ by \eqref{eq.energy_diss}.
Suppose that there exists a time $t_0 \in (0,T_m)$ such that $u(t_0) \in \mathcal M^-_L$.
Then, since $J_L (u(\cdot))$ is continuous on $[0,T_m)$, there exists a time $t_1 \in [0,t_0)$ such that $J_L (u(t_1)) = 0$.
Hence we see from \eqref{eq.l_L2} and \eqref{eq.ei} that
\[
l_L \le E_L (u(t_1))\le E_L (u_0).
\]
However this contradicts the assumption $E_L (u_0) <l_L$.
Thus
$u(t)\in \mathcal M^{+}_L$ for any $t \in [0,T_m)$.
Similarly, the case of $\mathcal M^-_L$ is also proved.
The proof of Lemma \ref{lem:Mpm} is finished.
\end{proof}

\begin{lem}\label{lem:ET}
If $u_0\in \mathcal M^+_L$, then there exists $\delta>0$ such that
\begin{equation}\label{eq.ET}
J_L (u(t)) \ge \delta \|u(t)\|_{\mathcal E(L)}^2
\end{equation}
for any $t\in [0,T_m)$.
\end{lem}

\begin{proof}
Since $E_L (u_0) < l_L$, there exists $\delta_0>0$ such that
\begin{equation}\label{eq.ass1}
E_L (u_0) \le (1-\delta_0) l_L.
\end{equation}
Consider the function
\[
F(y) := \frac12 y - \frac{S_{p+1}^{p+1}}{p+1} y ^{\frac{p+1}{2}}, \quad y\ge0.
\]
It is readily seen that
$F'(y)=0$ if and only if $y=y_C$, where
\[
y_C:= S_{p+1}^{-\frac{2(p +1)}{p-1}}.
\]
Then we see that
\begin{equation}\label{eq.y_C}
F(y_C)  = \frac{p-1}{2(p +1)} S_{p+1}^{-\frac{2(p +1)}{p-1}} = l_L
\quad \text{and}\quad
F''(y_C)  < 0.
\end{equation}
Hence it follows from \eqref{eq.Sobolev1}, \eqref{eq.ei},
\eqref{eq.ass1} and \eqref{eq.y_C} that
\[
F(\|u(t)\|_{\mathcal E(L)}) \le E_L (u(t)) \le E_L (u_0) \le (1-\delta_0) l_L = (1-\delta_0) F(y_C).
\]
for any $t \in [0,T_m)$.
Note that
\begin{equation}\label{eq.y_C2}
\|u(t)\|_{\mathcal E(L)} < y_C \quad \text{for any $t \in [0,T_m)$,}
\end{equation}
since
\[
\|u(t)\|_{\mathcal E(L)}
\le \frac{2(p+1)}{p-1} E_L(u(t))
<  \frac{2(p+1)}{p-1} l_L = S_{p+1}^{-\frac{2(p +1)}{p-1}} = y_C.
\]
Since $F$ is strictly increasing on $(0, y_C)$, there exists $\delta_1>0$ such that
\begin{equation}\label{eq.y_0}
\|u(t)\|_{\mathcal E(L)} \le (1-\delta_1)y_C
\end{equation}
for any $t \in [0,T_m)$.
Next, we consider the function
\[
G(y) := y - S_{p+1}^{p+1} y^\frac{p+1}{2}.
\]
Then
$G(y)=0$ if and only if $y=0$ or $y=y_C$.
Furthermore, $G'(0) = 1$ and $G'(y_C)=-(p-1)/2$.
Hence
\begin{equation}\label{eq.G}
G(y) \ge C\min{\{y,y_C-y\}}
\end{equation}
for any $0<y<y_C$.
Therefore, noting \eqref{eq.y_C2}, and taking $y = \|u(t)\|_{\mathcal E(L)}$, we deduce from \eqref{eq.y_0} and \eqref{eq.G} that
\[
J_L(u(t))
\ge G(\|u(t)\|_{\mathcal E(L)})
\ge C\min{\{\|u(t)\|_{\mathcal E(L)},y_C-\|u(t)\|_{\mathcal E(L)}\}}
\ge C \delta_1 \|u(t)\|_{\mathcal E(L)}
\]
for any $t \in [0,T_m)$.
Thus \eqref{eq.ET} is proved. The proof of Lemma \ref{lem:ET} is complete.
\end{proof}

\begin{lem}\label{lem:J_neg}
If $u_0 \in \mathcal M^-_L$, then
\[
J_L (u(t)) < -(p+1) \{l_L - E_L(u(t))\}
\]
for any $t \in (0,T_m)$.
\end{lem}

\begin{proof}
Let $t \in [0,T_m)$ be fixed.
By Lemma \ref{lem:Mpm}, we have
\begin{equation}\label{eq.M_-}
J_L (u(t)) < 0\quad \text{for any }t \in [0,T_m).
\end{equation}
Define the function
\[
K(\lambda) := E_L (e^{\lambda}u(t)),\quad \lambda\in\mathbb R.
\]
Then we calculate
\begin{equation}\label{eq.K'}
K'(\lambda)= e^{2\lambda}\|u(t)\|_{\mathcal E(L)}^2
- e^{(p+1)\lambda} \|u(t)\|_{L^{p+1}(\Omega)}^{p+1},
\end{equation}
\[
K''(\lambda)= 2e^{2\lambda}\|u(t)\|_{\mathcal E(L)}^2
- (p+1)e^{(p+1)\lambda} \|u(t)\|_{L^{p+1}(\Omega)}^{p+1}.
\]
Hence
\begin{equation}\label{eq.key}
K''(\lambda) - (p+1)K'(\lambda) = -(p-1)e^{2\lambda}\|u(t)\|_{\mathcal E(L)}^2 <0
\end{equation}
for any $\lambda\in\mathbb R$, since $p>1$.
We note from \eqref{eq.M_-} and \eqref{eq.K'} that  $K'$ is continuous in $\lambda$ and
\[
K'(0)=J_L (u(t)) < 0, \quad K'(\lambda)>0
\quad \text{for $-1 \ll \lambda<0$}.
\]
Then there exists $\lambda_0 <0$ such that $K'(\lambda_0) = 0$, which implies that $e^{\lambda_0}u (t)\in \mathcal N_L$ and $K(\lambda_0) \ge l_L$.
Integrating the inequality \eqref{eq.key} for the interval $(\lambda_0,0]$, we have
\[
K'(0) - K'(\lambda_0) < (p+1)( K(0) - K(\lambda_0)).
\]
From the above, we obtain
\[
J_L(u(t))
= K'(0) - K'(\lambda_0) < (p+1)( K(0) - K(\lambda_0))
\le (p+1)\{ E_L(u(t)) - l_L\}.
\]
Thus we conclude Lemma \ref{lem:J_neg}.
\end{proof}

%%%%%%%%%%%%%%%%%%%%%%%%%%%%%%%%%%%%%%
%%%%%%%%%%%%%%%% Section 5 %%%%%%%%%%%%%%%%
%%%%%%%%%%%%%%%%%%%%%%%%%%%%%%%%%%%%%%
\section{Proofs of Theorems \ref{thm:GB-sub} and \ref{thm:GB-cri}}\label{sec:5}
%We may assume that $E_L (u_0) < l_L$ without loss of generality.
%In fact, when $E_L (u_0) = l_L$ and $J_L(u_0)\not =0$,
%the problem is reduced to the case $E_L (u_0) < l_L$ (see Remarks \ref{rem:equal} and \ref{rem:equal2}).\\

First we prove Theorem \ref{thm:GB-sub}.

\begin{proof}[Proof of {\rm (i)} in Theorem \ref{thm:GB-sub}]
Let $u_0 \in  \mathcal M^+_L$.
By the definitions of $E_L$ and $J_L$, we calculate
\[
J_L (u(t))= -\frac{p-1}{2} \|u(t)\|_{H^1(L)}^2 + (p +1) E_L (u(t)).
\]
Since $J_L (u(t)) \ge 0$ for any $t \in [0,T_m)$ by Lemma \ref{lem:Mpm},
we have
\[
\|u(t)\|_{H^1(L)}^2
\le
\frac{2(p +1)}{p-1} E_L (u(t))
\le
\frac{2(p +1)}{p-1} E_L (u_0)
\]
for any $t \in [0,T_m)$.
Then $T_m=+\infty$ by (iii) in Proposition \ref{thm:LW-sub}.
Furthermore, we have
\[
\|u(t)\|_{L^2(\Omega)}^2
+
2 \int_0^t
J_L (u(s))\,ds
=
\|u_0\|_{L^2(\Omega)}^2
\]
for any $t>0$ by \eqref{eq.J_L2}.
Then we find from Lemma \ref{lem:ET} that
\[
2\delta \|u\|_{L^2(\mathbb R_+; H^1(L))}^2
\le \sup_{t\ge0} \left(\|u(t)\|_{L^2(\Omega)}^2 +2\delta \|u\|_{L^2((0,t); H^1(L))}^2\right) \le \|u_0\|_{L^2(\Omega)}^2,
\]
which proves that
\[
\lim_{t \to \infty} \|u(t)\|_{H^1(L)} = 0.
\]
Thus we conclude the assertion (i) in Theorem \ref{thm:GB-sub}.
\end{proof}

The proof of (ii) is done by the argument of proof of Proposition 6.1 in \cite{GR-2018}.
For completeness, we give the proof.

\begin{proof}[Proof of {\rm (ii)} in Theorem \ref{thm:GB-sub}]
Let $u_0 \in  \mathcal M^-_L$.
Define
\[
I(t) := \int_0^t \|u(s)\|_{L^2(\Omega)}^2\,ds + A,\quad t\in [0,T_m),
\]
where $A>0$, which is chosen later.
Then
\[
I'(t) = \|u(t)\|_{L^2(\Omega)}^2
\quad \text{and} \quad
I''(t) = -2 J_L(u(t)).
\]
By Schwarz' inequality, we estimate
\[
\begin{split}
I'(t)^2 & = \left(\| u_0\|_{L^2(\Omega)}^2 + 2 \mathrm{Re}\int_0^t (u,u_t)_{L^2(\Omega)}\,ds\right)^2\\
& \le (1+\varepsilon^{-1}) \| u_0\|_{L^2(\Omega)}^4 + 4(1+\varepsilon)\left(\int_0^t (u,u_t)_{L^2(\Omega)}\,ds\right)^2\\
& \le (1+\varepsilon^{-1}) \| u_0\|_{L^2(\Omega)}^4 + 4(1+\varepsilon)
\left(\int_0^t \| u(s)\|_{L^2(\Omega)}^2\,ds\right) \left(\int_0^t \| u_t (s)\|_{L^2(\Omega)}^2\,ds\right)
\end{split}
\]
for any $\varepsilon>0$,
where $(\cdot,\cdot)_{L^2(\Omega)}$ stands for the inner product of $L^2(\Omega)$.
Furthermore, it follows from Lemma \ref{lem:J_neg} that
\[
I''(t) \ge 2(p+1) \{ l_L - E_L(u(t))\}
\ge
2(p+1) \left( l_L - E_L(u_0) + \int_0^t \| u_t (s)\|_{L^2(\Omega)}^2\,ds \right).
\]
Let $\alpha>0$. By summarizing the above estimates,
we have
\begin{equation}\label{eq.keypoint}
\begin{split}
I''(t)&I(t) - (1+\alpha)I'(t)^2 \\
\ge &\,
2(p+1) \left( l_L - E_L(u_0) + \int_0^t \| u_t (s)\|_{L^2(\Omega)}^2\,ds \right)
\left(  \int_0^t \|u(s)\|_{L^2(\Omega)}^2\,ds + A \right)\\
- &\, 4(1+\alpha)(1+\varepsilon) \left(\int_0^t \| u(s)\|_{L^2(\Omega)}^2\,ds\right)
\left(\int_0^t \| u_t (s)\|_{L^2(\Omega)}^2\,ds\right) \\
- &\, (1+\alpha)(1+\varepsilon^{-1}) \| u_0\|_{L^2(\Omega)}^4
\end{split}
\end{equation}
for any $ t\in (0,T_m)$ and $\varepsilon>0$.
Noting that $l_L - E_L(u_0)$ is a positive constant, and
choosing $\alpha, \varepsilon$ sufficiently small and $A$ sufficiently large,
we can ensure that
\[
I''(t)I(t) - (1+\alpha)I'(t)^2 > 0
\]
for any $ t\in (0,T_m)$.
This is equivalent to
\[
\frac{d}{dt}\left(
\frac{I'(t)}{I(t)^{\alpha+1}} \right)
>0,
\]
which implies that
\[
\frac{I'(t)}{I(t)^{\alpha+1}} > \frac{I'(0)}{I(0)^{\alpha+1}} = \frac{\| u_0\|_{L^2(\Omega)}^2}{A^{\alpha+1}}=: a
\]
for any $ t\in (0,T_m)$. Integrating the above inequality gives
\[
\frac{1}{\alpha} \left(
\frac{1}{I(0)^\alpha }- \frac{1}{I(t)^\alpha}
\right)
> at.
\]
Hence
\[
I(t)^\alpha >
\frac{I(0)^\alpha}{1-I(0)^\alpha\alpha a t} \to +\infty
\]
as $t \to 1/(I(0)^\alpha\alpha a )= A/(\alpha \|u_0\|_{L^2(\Omega)}^2 ) =:\tilde{t}$ ($<+\infty$). This shows that
\begin{equation}\label{eq.blowup}
\limsup_{t\to \tilde{t}-}\|u(t)\|_{L^2(\Omega)} = +\infty.
\end{equation}
If $T_m = +\infty$, then the solution $u$ must satisfy
\[
u \in C([0,T]; L^2(\Omega))\quad \text{for any $T>0$}
\]
by Proposition \ref{thm:LW-sub}.
However this contradicts \eqref{eq.blowup}.
Thus we prove that $T_m < +\infty$.
The proof of (ii) in Theorem \ref{thm:GB-sub} is complete.
\end{proof}

Next we prove Theorem \ref{thm:GB-cri}.

\begin{proof}[Proof of {\rm (i)} in Theorem \ref{thm:GB-cri}]
Let $u_0 \in \dot H^1(L)$ such that $J_L(u_0)\ge0$ and $\|e^{-tL}u_0\|_{S(\mathbb R_+)} $ is sufficiently small. 
The first part, i.e., $T_m = + \infty$,
is proved by (vi) in Proposition~\ref{thm:LW-cri}. Hence it suffices to prove the latter part:
\begin{equation*}\label{eq.diss}
\lim_{t\to\infty} \|u(t)\|_{\dot{H}^1(L)}=0.
\end{equation*}

Let $\varepsilon>0$ be fixed.
The solution $u$ to \eqref{eq.NLH-cri} is written as
\[
\begin{split}
u(t)  & =
e^{-tL} u_0 + \int_{0}^\tau e^{-(t-s)L}|u(s)|^{\frac{4}{d-2}} u(s)\, ds
+
\int_{\tau}^t e^{-(t-s)L}|u(s)|^{\frac{4}{d-2}} u(s)\, ds\\
& =:  I(t) + I\hspace{-1pt}I(t) + I\hspace{-1pt}I\hspace{-1pt}I(t)
\end{split}
\]
for $0<\tau <t$. By density,
there exists $v_\varepsilon \in L^1(\Omega) \cap L^2(\Omega)$ such that
\[
\| L^\frac12 u_0 - v_\varepsilon\|_{L^2(\Omega)} < \frac{\varepsilon}{2}.
\]
Then it follows from (i) in Lemma \ref{lem:Str} that
\[
\begin{split}
\| I(t) \|_{\dot H^1(L)} & \le \| L^\frac12 u_0 - v_\varepsilon\|_{L^2(\Omega)} + C t^{-\frac{d}{4}}\|v_\varepsilon\|_{L^1(\Omega)}\\
& \le\frac{\varepsilon}{2} + C t^{-\frac{d}{4}}\|v_\varepsilon\|_{L^1(\Omega)}
\end{split}
\]
for any $t>0$. Hence there exists a time $t_1 = t_1(\varepsilon)>0$ such that
\begin{equation}\label{eq.first}
\| I(t) \|_{\dot H^1(L)} \le \varepsilon\quad \text{for any $t > t_1$.}
\end{equation}
As to the second term $I\hspace{-1pt}I(t)$,
we write
\[
I\hspace{-1pt}I(t) =  e^{-(t-\tau)L} \int_{0}^\tau e^{-(\tau-s)L}|u(s)|^{\frac{4}{d-2}} u(s)\, ds.
\]
Since
\[
w(\tau):= \int_{0}^\tau e^{-(\tau-s)L}|u(s)|^{\frac{4}{d-2}} u(s)\, ds \in \dot H^1(L)
\]
for $0<\tau <t$ by the same argument as in \eqref{eq.s3-1},
we can apply the same argument as in $I(t)$ to $I\hspace{-1pt}I(t)$, and hence,
there exists a time $t_2 = t_2(\varepsilon)>0$ such that
\begin{equation}\label{eq.second}
\| I\hspace{-1pt}I(t) \|_{\dot H^1(L)} \le \varepsilon\quad \text{for any $t > t_2$.}
\end{equation}
As to the third term $I\hspace{-1pt}I\hspace{-1pt}I(t)$, again applying the same argument as in \eqref{eq.s3-1}, we estimate
\[
\| I\hspace{-1pt}I\hspace{-1pt}I(t) \|_{\dot H^1(L)} \le C\|u\|_{S((\tau, t))}^{\frac{d+2}{d-2}}.
\]
Since $\|u\|_{S(\mathbb R_+)}<\infty$ by (vi) in Proposition \ref{thm:LW-cri},
there exist $\tau_0 = \tau_0(\varepsilon) >0$ such that
\begin{equation}\label{eq.third}
\| I\hspace{-1pt}I\hspace{-1pt}I(t) \|_{\dot H^1(L)} \le C\|u\|_{S((\tau, t))}^{\frac{d+2}{d-2}} < \varepsilon \quad \text{for any $t>\tau >\tau_0$}.
\end{equation}
By combining \eqref{eq.first}--\eqref{eq.third}, we conclude that
\[
\lim_{t\to\infty}\|u(t)\|_{\dot H^1(L)}=0.
\]
The proof of (i) in Theorem \ref{thm:GB-cri} is finished.
\end{proof}

\begin{proof}[Proof of {\rm (ii)} in Theorem \ref{thm:GB-cri}]
Let $u_0 \in \mathcal M_L^-$. Suppose $T_m=T_m(u_0) = +\infty$.
We take $\chi_R \in C^\infty_0(\mathbb R^n)$ such that
\[
\chi_R (x)
=
\begin{cases}
1,\quad &|x|\le R,\\
0, & |x|\ge R+1
\end{cases}
\]
for $R>0$. Then, by H\"older's inequality and  the Sobolev embedding theorem \eqref{eq.Sobolev2}, we have
\begin{equation}\label{eq.1}
\| \chi_R u(t) \|_{L^2(\Omega)} \le C_R \|u(t) \|_{L^{p^*+1}(\Omega)} \le C_R' \| u(t) \|_{\dot H^1(L)},
\end{equation}
and hence $\chi_R u (t) \in L^2(\Omega)$ for $t>0$.
We define
\[
I_R(t) := \int_0^t \| \chi_R u(s) \|_{L^2(\Omega)}^2 \,ds + A,\quad t\ge0,
\]
where $R>0$ and $A>0$, which are chosen later.
Then
\[
I_R'(t)=\| \chi_R u(t) \|_{L^2(\Omega)}^2,
\]
\begin{align*}
I_R''(t) & = 2 (\chi_R u(t), \chi_R u_t(t))_{L^{2}(\Omega)}\\
& = 2 (u(t), u_t(t))_{L^{2}(\Omega)} + 2((\chi_R^2-1)u(t), u_t(t))_{L^{2}(\Omega)}\\
& = -2 J_L(u(t)) + 2((\chi_R^2-1)u(t), u_t(t))_{L^{2}(\Omega)}. 
\end{align*}
Here we note that
\begin{equation}\label{eq.sec5-1}
| (u(t), u_t(t))_{L^{2}(\Omega)} | < \infty
\end{equation}
for almost everywhere $t>0$. Indeed, in a similar argument to \eqref{eq.s3-1}, 
we estimate
\[
\| u_t \|_{L^{\infty}((0,T);\dot H^{-1}(L))}
\le C \left(\| u_0 \|_{\dot{H}^1(L)}
+
\|u \|_{L^{\infty}((0,T);\dot{H}^1(L))}^{\frac{d+2}{d-2}}
+
\|u\|_{S((0,T))}^{\frac{d+2}{d-2}}
\right) < \infty.
\]
Hence
\begin{equation}\label{eq.finite}
\sup_{t\in(0,T)}| (u(t), u_t(t))_{L^{2}(\Omega)} | 
\le \| u \|_{L^{\infty}((0,T);\dot H^1(L))} \| u_t \|_{L^{\infty}((0,T);\dot H^{-1}(L))} <\infty,
\end{equation}
which shows \eqref{eq.sec5-1}. 
Therefore, in a similar way to proof of (ii) in Theorem \ref{thm:GB-sub}, we have
\begin{equation}\label{eq.keypoint2}
I''_R(t)I_R(t) - (1+\alpha)I'_R(t)^2 > 0
\end{equation}
for almost everywhere $t >0$, if $R$ and $A$ are sufficiently large and $\alpha$ is sufficiently small.
In fact, we estimate $I_R'(t)^2$ as
\begin{align*}
I_R'(t)^2 & = \left(\| \chi_Ru_0\|_{L^2(\Omega)}^2 + 2 \mathrm{Re}\int_0^t (\chi_Ru,\chi_Ru_t)_{L^2(\Omega)}\,ds\right)^2\\
& \le (1+\varepsilon^{-1}) \| \chi_Ru_0\|_{L^2(\Omega)}^4 + 4(1+\varepsilon)\left(\int_0^t (\chi_Ru,\chi_Ru_t)_{L^2(\Omega)}\,ds\right)^2\\
& \le (1+\varepsilon^{-1}) \| u_0\|_{L^2(\Omega)}^4 + 4(1+\varepsilon)
\left(\int_0^t \| \chi_Ru(s)\|_{L^2(\Omega)}^2\,ds\right) \left(\int_0^t \| \chi_Ru_t (s)\|_{L^2(\Omega)}^2\,ds\right)
\end{align*}
for any $\varepsilon >0$, and estimate $I_R''(t)$ from below as
\begin{align*}
I_R''(t) &\ge 2(p+1) \{ l_L - E_L(u(t))\}+ 2((\chi_R^2-1)u(t), u_t(t))_{L^{2}(\Omega)}\\
&\ge
2(p+1) \left( l_L - E_L(u_0) + \int_0^t \| u_t (s)\|_{L^2(\Omega)}^2\,ds \right)+ 2((\chi_R^2-1)u(t), u_t(t))_{L^{2}(\Omega)}\\
&\ge
2(p+1) \left( l_L - E_L(u_0) + \int_0^t \| \chi_R u_t (s)\|_{L^2(\Omega)}^2\,ds \right)+ 2((\chi_R^2-1)u(t), u_t(t))_{L^{2}(\Omega)}.
\end{align*}
Let $\alpha>0$. By summarizing the above estimates,
we have
\[
\begin{split}
&I_R''(t)I_R(t) - (1+\alpha)I_R'(t)^2 \\
\ge &\,
2(p+1) \left( l_L - E_L(u_0) + \int_0^t \| \chi_Ru_t (s)\|_{L^2(\Omega)}^2\,ds \right)
\left(  \int_0^t \|\chi_R u(s)\|_{L^2(\Omega)}^2\,ds + A \right)\\
- &\, 4(1+\alpha)(1+\varepsilon) \left(\int_0^t \| \chi_R u(s)\|_{L^2(\Omega)}^2\,ds\right)
\left(\int_0^t \| \chi_R u_t (s)\|_{L^2(\Omega)}^2\,ds\right) \\
- &\, (1+\alpha)(1+\varepsilon^{-1}) \| u_0\|_{L^2(\Omega)}^4
- 2|((\chi_R^2-1)u(t), u_t(t))_{L^{2}(\Omega)}|
\end{split}
\]
for almost everywhere $t>0$.
Here it is ensured by \eqref{eq.finite} that
\[
|((\chi_R^2-1)u(t), u_t(t))_{L^{2}(\Omega)}| \to 0 \quad \text{as }R\to + \infty.
\]
Hence, choosing $R$, $A$ sufficiently large and $\varepsilon$, $\alpha$ sufficiently small,
we obtain \eqref{eq.keypoint2}
by the same argument as in \eqref{eq.keypoint}.
Therefore, by the same argument as in \eqref{eq.blowup}, we have
\[
\limsup_{t\to \tilde{t}-} \|\chi_R u(t)\|_{L^2(\Omega)} = +\infty,
\]
where $\tilde{t} := A/(\alpha \|\chi_R u_0\|_{L^2(\Omega)}^2 ) < \infty$.
Hence we find from \eqref{eq.1} that
\begin{equation*}\label{eq.blowup2}
\limsup_{t\to \tilde{t}-} \| u(t)\|_{\dot H^1(L)} = +\infty.
\end{equation*}
However this contradicts $T_m=T_m(u_0) = +\infty$, i.e.,
\[
u \in C([0,T]; \dot H^1(L))\quad \text{for any $T>0$}
\]
by Proposition \ref{thm:LW-cri}.
Thus we conclude that $T_m < + \infty$.
The proof of (ii) in Theorem~\ref{thm:GB-cri} is complete.
\end{proof}

%%%%%%%%%%%%%%%%%%%%%%%%%%%%%%%%%%%%%%
%%%%%%%%%%%%%% Appendix A %%%%%%%%%%%%%%%%%%
%%%%%%%%%%%%%%%%%%%%%%%%%%%%%%%%%%%%%%
\appendix
\section{Definitions of Sobolev spaces associated with $L$}
\label{App:A}

In this appendix we show well-definedness and completeness of homogeneous Sobolev spaces $\dot H^s(L)$. 
Although this is mentioned in \cite{IMT-Besov}, 
we give more details here.

The definition is based on \cite{IMT-Besov},
in which the theory of homogeneous Besov spaces $\dot B^s_{p,q}(-\Delta_D)$ associated with the Dirichlet Laplacian $-\Delta_D$
is established on an arbitrary open set of $\mathbb R^d$. 
The key points to define $\dot B^s_{p,q}(-\Delta_D)$ are the following two facts:
\begin{itemize}
\item[(i)] $L^p$-boundedness of spectral multiplier operators $\phi(-\theta \Delta_D)$ for $1\le p\le \infty$:
\begin{equation}\label{eq.spec-Lp}
\sup_{\theta >0} \|\phi (-\theta \Delta_D)\|_{L^p(\Omega) \to L^p(\Omega)} < \infty,
\end{equation}
provided $\phi \in C^\infty_0((0,\infty))$;
\item[(ii)] zero is not an eigenvalue of $-\Delta_D$.
\end{itemize}
If $L$ is a non-negative and self-adjoint operator on $L^2(\Omega)$ satisfying (i) and (ii),
then we can apply 
the argument of \cite{IMT-Besov} to $L$,
and hence, we can define the homogeneous Besov spaces $\dot B^s_{p,q}(L)$ associated with $L$
whose norms are given by
\[
\| f \|_{\dot B^s_{p,q}(L)} :=
\left\{ \sum_{j=-\infty}^\infty
\left(2^{sj} \|\phi_j(\sqrt{L}) f\|_{L^p(\Omega)}\right)^q
\right\}^{\frac{1}{q}}
\]
for $s\in \mathbb R$ and $1\le p,q\le \infty$,
where $\{\phi_j\}_j$ is the Littlewood-Paley dyadic decomposition.
Then the homogeneous Sobolev space $\dot H^s(L)$ is defined by
\begin{equation}\label{eq.homo-S}
\dot H^s(L) := \dot B^s_{2,2}(L),\quad s\in \mathbb R.
\end{equation}
It is proved that $\dot H^s(L)$ is complete (see Theorem 2.5 in \cite{IMT-Besov}), and
the definition \eqref{eq.homo-S} is equivalent to (ii) in Definition \ref{def:Sobolev}.
Therefore, in order to define $\dot H^s(L)$, it is sufficient to show that $L$ satisfies (i) and (ii) under Assumption~B. \\

As to (i),
the spectral multiplier theorem is already established for non-negative self-adjoint operators
with Gaussian upper bound \eqref{eq.Gauss} (see Duong, Ouhabaz and Sikora \cite{DOS-2002} and also Bui, D'Ancona and Nicola \cite{BDN-arxiv}).
From this theorem, we obtain $L^p$-boundedness \eqref{eq.spec-Lp} for $L$ under Assumption B.\\

As to (ii), we have the following:

\begin{prop}\label{prop:ev}
Let $L$ be a self-adjoint operator on $L^2(\Omega)$ satisfying
\begin{equation}\label{eq.decay}
\| e^{-tL} f \|_{L^\infty(\Omega)} \to 0\quad \text{as }t\to\infty
\end{equation}
for any $f\in L^2(\Omega)$.
Then zero is not an eigenvalue of $L$.
\end{prop}

\begin{proof}
Suppose that zero is an eigenvalue of $L$, i.e., there exists a function $f_0 \in \mathcal D(L)\setminus \{0\}$
such that $Lf_0 = 0$. Then
\[
\partial_t e^{-tL} f_0 = L e^{-tL} f_0  = e^{-tL} L f_0 = 0
\]
for any $t>0$, which implies that $e^{-tL} f_0$ is a constant in $t>0$.
Taking account of \eqref{eq.decay}, we have
$
e^{-tL} f_0 = 0
$,
and hence,
\begin{equation}\label{eq.A1}
(I+L)^{-1} f_0 = \int_0^\infty
e^{-t}e^{-tL}f_0 \, dt = 0
\end{equation}
almost everywhere in $\Omega$.
Since $-1$ belongs to the resolvent set of $L$, the operator $(I+L)^{-1}$ is injective from
$\mathcal D(L)$ to $L^2(\Omega)$. Therefore we deduce from \eqref{eq.A1} that $f_0=0$.
However this contradicts the fact that zero is an eigenvalue of $L$.
Thus we conclude that zero is not an eigenvalue of $L$.
\end{proof}

Hence it follows from Proposition \ref{prop:ev} that zero is not an eigenvalue of $L$ under Assumption B.
Thus $\dot H^s(L)$ is well-defined and complete.

%%%%%%%%%%%%%%%%%%%%%%%%%%%%%%%%%%%%%%
%%%%%%%%%%%%%% Appendix B %%%%%%%%%%%%%%%%%%
%%%%%%%%%%%%%%%%%%%%%%%%%%%%%%%%%%%%%%

\section{Proof of Proposition \ref{prop:Sobolev}}\label{App:B}

In this appendix we prove Proposition \ref{prop:Sobolev}.
We first prove the assertion (i). By using the formula
\[
(I+L)^{-\frac12}
=
\frac{1}{2\sqrt{\pi}}
\int_0^\infty t^{-\frac12} e^{- t}e^{-tL}\,dt,
\]
we deduce from \eqref{eq.L2Lq} that
\[
\begin{split}
\| (I+L)^{-\frac12} g\|_{L^{p+1}(\Omega)}
& \le C \int_0^\infty  t^{-\frac12} e^{- t}
\|e^{-tL} g\|_{L^{p+1}(\Omega)}\,dt\\
& \le
C \|g\|_{L^2(\Omega)} \int_0^\infty t^{-\frac12-\frac{d}{2}(\frac12 - \frac{1}{p+1})} e^{- (1-\omega)t}\,dt\\
& \le C \|g\|_{L^2(\Omega)}
\end{split}
\]
for any $g\in L^2(\Omega)$, since $2 < p+1 < 2d/(d-2)$.
Hence we conclude that
\[
\| f \|_{L^{p+1}(\Omega)}
\le
C \|(I+L)^{\frac12} f \|_{L^2(\Omega)}
\le C \|f\|_{H^1(L)}
\]
for any $f \in H^1(L)$, which concludes the assertion (i).

As to the assertion (ii), under the assumption \eqref{eq.Gauss}, 
the same arguments as Propositions~3.2 and 3.3 in \cite{IMT-Besov} allow us to obtain the following embedding relations:
\[
\dot H^{1}(L)=\dot B^1_{2,2}(L) \hookrightarrow\dot B^0_{\frac{2d}{d-2},2}(L)
\hookrightarrow L^{\frac{2d}{d-2}}(\Omega).
\]
Therefore we conclude the assertion (ii) from the above embeddings. 
The proof of Proposition \ref{prop:Sobolev} is complete.

%%%%%%%%%%%%%%%%%%%%%%%%%%%%%%%%%%%%%%
%%%%%%%%%%%%%% Appendix C %%%%%%%%%%%%%%%%%%
%%%%%%%%%%%%%%%%%%%%%%%%%%%%%%%%%%%%%%

\section{Proof of \eqref{eq.l_L2}}\label{App:C}
In this appendix we show the characterization \eqref{eq.l_L2} of $l_{L}$ by the best constants $S_{p+1}$ of the Sobolev inequalities. 
First we show that
\begin{equation}\label{eq.less}
\frac{p-1}{2(p+1)} S_{p+1}^{-\frac{2(p+1)}{p-1}} \le l_L.
\end{equation}
Let $f \in \mathcal N_L$. Then
we write the Pohozaev identity with $L$:
\[
E_L(f) =
\frac{p-1}{2(p+1)}\|f\|_{\mathcal E(L)}^2.
\]
By the Sobolev inequality \eqref{eq.Sobolev1} and \eqref{eq.Sobolev2}, we estimate
\[
\|f\|_{\mathcal E(L)}^2 = \| f\|_{L^{p+1}(\Omega)}^{p+1}
\le (S_{p+1} \|f\|_{\mathcal E(L)})^{p+1},
\]
which implies that
\[
S_{p+1}^{-\frac{2(p+1)}{p-1}}
\le
\|f\|_{\mathcal E(L)}^2.
\]
Hence, combining the estimates obtained now, we have
\[
\frac{p-1}{2(p+1)}S_{p+1}^{-\frac{2(p+1)}{p-1}}
\le E_L(f),
\]
which shows \eqref{eq.less} by taking the infimum of the right hand side over $f \in \mathcal N_L$.

Next we show \eqref{eq.l_L2}. Since
$
S_{p+1}^{-1} = \inf{\left\{\|f\|_{\mathcal E(L)} : f \in \mathcal E(L),\ \|f\|_{L^{p+1}(\Omega)}=1 \right\}}
$,
it follows that
there exists $\{f_\varepsilon\}_{\varepsilon>0} \subset \mathcal E(L)$
such that
\begin{equation}\label{eq.inf_A}
\|f_\varepsilon\|_{\mathcal E(L)} < S_{p+1}^{-1} + \varepsilon
\quad \text{and}\quad
\|f_\varepsilon\|_{L^{p+1}(\Omega)}=1.
\end{equation}
Furthermore, there exists $\lambda_\varepsilon >0$ such that
$\lambda_\varepsilon f_\varepsilon \in \mathcal N_L$, and we put $g_\varepsilon := \lambda_\varepsilon f_\varepsilon$ for $\varepsilon>0$.
Then, noting the equality in \eqref{eq.inf_A} and
$g_\varepsilon \in \mathcal N_L$, we write
\[
\|f_\varepsilon\|_{\mathcal E(L)}
= \frac{\|g_\varepsilon\|_{\mathcal E(L)}}{\|g_\varepsilon\|_{L^{p+1}(\Omega)}}
= \|g_\varepsilon\|_{\mathcal E(L)}^{\frac{p-1}{p+1}}
=
\bigg\{\frac{2(p+1)}{p-1} E_L(g_\varepsilon)\bigg\}^{\frac{p-1}{2(p+1)}}
\]
Hence, combining the inequality in \eqref{eq.inf_A} and the above equality,
we have
\begin{equation}\label{eq.222}
E_L(g_\varepsilon)
<
\frac{p-1}{2(p+1)} (S_{p+1}^{-1}+ \varepsilon)^{\frac{2(p+1)}{p-1}}
\end{equation}
for any $\varepsilon>0$. Suppose that the right hand side of \eqref{eq.less} is strictly less than $l_L$.
Then
it follows from \eqref{eq.222} that
\[
E_L(g_\varepsilon)
<
l_L
\]
for sufficiently small $\varepsilon>0$.
This contradicts the definition of $l_L$, since $g_\varepsilon \in \mathcal N_L$.
Thus we conclude \eqref{eq.l_L2}.

%%%%%%%%%%%%%%%%%%%%%%%%%%%%%%%%%%%%%%
%%%%%%%%%%%%%% Appendix D %%%%%%%%%%%%%%%%%%
%%%%%%%%%%%%%%%%%%%%%%%%%%%%%%%%%%%%%%

\section{Proof of Proposition \ref{thm:LW-cri}}\label{App:D}
The proof of Proposition \ref{thm:LW-cri} is based on the fixed point argument 
(see \cite{CW-1990} and also \cite{KM-2006}). 
In this appendix, let us give the proof for completeness.\\ 

First we prove the assertion (i), i.e., 
the existence of solutions in the sense of Definition~\ref{def:2}.
For this purpose, we show the following lemma.

\begin{lem}[Theorem 2.5 in \cite{KM-2006}]\label{lem:local}
Let $A>0$ and $u_{0} \in \dot H^{1}(L)$ with $\|u_{0}\|_{\dot H^{1}(L)}\le A$. 
Then there exists a constant $\delta=\delta(A)>0$ such that 
if $\|e^{-tL}u_{0}\|_{S(I)}\le \delta$, then there exists 
a unique solution $u\in C(I;\dot H^{1}(L))$ to \eqref{eq.NLH-cri} with 
$u(0)=u_{0}$,
\begin{equation}\label{eq.D1}
\| u_t \|_{L^2(I\times\Omega)} < \infty
\quad \text{and}\quad 
\| u \|_{S(I)}\le 2\delta.
\end{equation}
\end{lem}

\begin{proof}
Let $T>0$ and we assume that $I=[0,T)$ without loss of generality. 
Define the map $\Phi_{u_0}$ by
\begin{equation}\label{eq.Phi}
\Phi_{u_0}[u](t):=
e^{-tL}u_0 +
\int_0^t e^{-(t-s)L} |u(s)|^{\frac{4}{d-2}}u(s)\,ds.
\end{equation}
Given $a,b>0$ to be chosen later, we define 
\[
X_{a,b} := \Big\{
u \text{ on } I\times\Omega 
:
\| u \|_{S(I)}
\le a,\ 
\| u_t \|_{L^2(I\times\Omega)} 
\le b
\Big\},
\]
equipped with the distance $\mathrm{d}:X_{a,b}\times X_{a,b}\to\mathbb R$ given by 
\begin{equation}\label{eq.distance}
\mathrm{d}(u,v)
:=\| u-v \|_{S(I)}.
\end{equation}
Then $(X_{a,b},\mathrm{d})$ is a complete metric space. 
Indeed, let $\{u_{n}\}_{n\in\mathbb N}$ be a Cauchy sequence in $(X_{a,b}, \mathrm{d})$. 
Then $\{u_{n}\}_{n\in\mathbb N}$ is also a Cauchy sequence in $S(I)$ and its limit $u \in S(I)$ satisfies 
\[
\|u\|_{S(I)} \le \liminf_{n\to\infty} \|u_{n}\|_{S(I)} \le a
\]
by Fatou's lemma. Moreover, since $\{\partial_{t}u_{n}\}_{n}$ is a bounded sequence in $L^{2}(I\times\Omega)$ with $\|\partial_{t}u_{n}\|_{L^{2}(I\times\Omega)}\le b$ for any $n\in\mathbb N$, it follows from the Banach-Alaoglu theorem that there exist a subsequence $\{u_{n(k)}\}_{k}$ and a function $v \in L^{2}(I\times\Omega)$ such that 
\[
u_{n(k)} \to v \quad \text{weakly* in }L^{2}(I\times\Omega)
\]
as $k\to\infty$, and again applying Fatou's lemma, we have
\[
\|v\|_{L^{2}(I\times\Omega)}\le \liminf_{k\to\infty} \|u_{n(k)}\|_{L^{2}(I\times\Omega)}
\le b.
\]
Finally, since $u_{n}\to u$ in $\mathscr D'(I\times\Omega)$ as $n\to\infty$, where 
$\mathscr D'(I\times\Omega)$ is the space of distributions on $I\times\Omega$, 
we have $\partial_{t}u_{n}\to \partial_{t}u$ in $\mathscr D'(I\times\Omega)$ as $n\to\infty$, i.e.,
\[
\langle \partial_{t}u_{n} - \partial_{t}u, \phi\rangle 
= 
\langle u_{n} - u, \partial_{t}\phi\rangle \to 0
\]
as $n\to\infty$ for any $\phi \in C^{\infty}_{0}(I\times\Omega)$. 
This implies that $v=\partial_{t}u$ by the uniqueness of limit in $\mathscr D'(I\times\Omega)$. Thus $(X_{a,b},\mathrm{d})$ is complete.

Let $u\in X_{a,b}$. 
Then, 
by using the space-time estimates in Lemma \ref{lem:Str} and the assumption on $u_{0}$, we obtain
\begin{equation}\label{eq.s3-1}
\| \Phi_{u_0}[u] \|_{S(I)} 
\le \|e^{-tL}u_{0}\|_{S(I)} 
+ C\|u\|_{S(I)}^{\frac{d+2}{d-2}}
\le 
\delta + Ca^{\frac{d+2}{d-2}}.
\end{equation}
Choosing $\delta = a/2$ and $Ca^{\frac{4}{d-2}}\le 1/2$, we have 
$\| \Phi_{u_0}[u] \|_{S(I)} \le a$. 
Next, writing $\Phi_{u_0}[u](t)$ as 
\[
\Phi_{u_0}[u](t) = e^{-tL}u_0 + \int_0^t e^{-s'L} |u(t-s')|^{\frac{4}{d-2}}u(t-s') \,ds'
\]
by the change $s \mapsto t-s'$, 
we estimate 
\begin{multline*}
\|\partial_{t} \Phi_{u_0}[u]\|_{L^2(I\times\Omega)}
\le 
\| e^{-tL} L^\frac12 u_0 \|_{L^2(I; \dot{H}^1(L))}
+ 
\|e^{-tL} |u_0|^{\frac{4}{d-2}}u_0 \|_{L^2(I\times\Omega)}\\
+
\left\|
\int_0^t e^{-s'L} \partial_t( |u(t-s')|^{\frac{4}{d-2}}u(t-s')) \,ds'
\right\|_{L^2(I\times\Omega)}.
\end{multline*}
As to the first term, we use 
the space-time estimate \eqref{eq.Str2} to get
\[
\| e^{-tL} L^\frac12 u_0 \|_{L^2(I; \dot{H}^1(L))}
\le C \|u_0 \|_{\dot{H}^1(L)}.
\]
As to the second term, it follows from the space-time estimate \eqref{eq.Str1} and 
the Sobolev inequality \eqref{eq.Sobolev2} that 
\[
\|e^{-tL} |u_0|^{\frac{4}{d-2}}u_0 \|_{L^2(I\times\Omega)}
\le 
C \||u_0|^{\frac{4}{d-2}}u_0 \|_{L^\frac{2d}{d+2}(\Omega)}
=
C \|u_0 \|_{L^\frac{2d}{d-2}(\Omega)}^{\frac{d+2}{d-2}}
\le 
C \|u_0 \|_{\dot{H}^1(L)}^{\frac{d+2}{d-2}}.
\]
As to the third term, by using the space-time estimate \eqref{eq.Str3}, we have
\[
\begin{split}
\left\|
\int_0^t e^{-s'L} \partial_t( |u(t-s')|^{\frac{4}{d-2}}u(t-s')) \,ds'
\right\|_{L^2(I\times\Omega)}
& \le C \| \partial_t( |u|^{\frac{4}{d-2}}u) \|_{L^\frac{2(d+2)}{d+4}(I\times\Omega)}\\
& \le C \| u \|_{S(I)}^{\frac{4}{d-2}} \|u_t\|_{L^2(I\times\Omega)}.
\end{split}
\]
By combining the above four estimates, we obtain 
\begin{equation}\label{eq.s3-2}
\begin{split}
\|\partial_{t} \Phi_{u_0}[u]\|_{L^2(I\times\Omega)}
& \le 
C 
\left(
\|u_0 \|_{\dot{H}^1(L)}
+
\|u_0 \|_{\dot{H}^1(L)}^{\frac{d+2}{d-2}}
+
\| u \|_{S(I)}^{\frac{4}{d-2}} \|u_t\|_{L^2(I\times\Omega)}
\right)\\
& \le 
C(A + A^{\frac{d+2}{d-2}}) + Ca^{\frac{4}{d-2}}b.
\end{split}
\end{equation}
Choose $b=2C(A + A^{\frac{d+2}{d-2}})$. Then 
$\|\partial_{t} \Phi_{u_0}[u]\|_{L^2(I\times\Omega)}\le b$.
Similarly, we also have
\[
\begin{split}
\mathrm{d}(\Phi_{u_0}[u],\Phi_{u_0}[v])
&\le C 
\max\left\{\|u\|_{S(I)}^{\frac{4}{d-2}}, \|v\|_{S(I)}^{\frac{4}{d-2}}\right\}
\mathrm{d}(u,v)\\
& \le C a^{\frac{4}{d-2}}\mathrm{d}(u,v)\\
& \le \frac12 \mathrm{d}(u,v)
\end{split}
\]
for $u,v \in X_{a,b}$, since $C a^{\frac{4}{d-2}}\le 1/2$. 
Therefore $\Phi_{u_0}$ is a contraction mapping from $X_{a,b}$ into itself. 
By Banach's fixed point theorem, we find $u \in X_{a,b}$ solving 
$\Phi_{u_0}[u]=u$. 
Moreover we see that $u \in C(I; \dot H^{1}(L))$, 
since 
\[
e^{-tL}u_{0}\in C(I; \dot H^{1}(L))
\quad \text{and}\quad 
\int_{0}^{\infty}e^{-(t-s)L}|u(s)|^{\frac{4}{d-2}}u(s)\, ds \in C(I; \dot H^{1}(L))
\]
by \eqref{eq.Str1} 
and \eqref{eq.Str4}, respectively. 
Thus we conclude Lemma \ref{lem:local}.
\end{proof}

\begin{proof}[Proof of {\rm (i)} in Proposition \ref{thm:LW-cri}]
We note from \eqref{eq.Str1} that 
\[
\|e^{-tL}u_{0}\|_{S(\mathbb R_{+})}\le C \|u_{0}\|_{\dot H^{1}(L)}<\infty.
\]
By Lemma \ref{lem:local}, 
for any $u_{0} \in \dot H^{1}(L)$, there exists an interval $I$ such that the hypotheses on $I$ in Lemma \ref{lem:local} is satisfied. 
Moreover, it follows from a similar argument to \eqref{eq.s3-2} that 
the solution $u$ obtained in Lemma \ref{lem:local} 
satisfies $Lu \in L^2(I\times\Omega)$, since
\[
\|Lu\|_{L^2(I\times\Omega)}
\le 
C 
\left(
\|u_0 \|_{\dot{H}^1(L)}
+
\|u_0 \|_{\dot{H}^1(L)}^{\frac{d+2}{d-2}}
+ 
\|u\|_{S(I)}^{\frac{d+2}{d-2}}
+
\| u \|_{S(I)}^{\frac{4}{d-2}} \|u_t\|_{L^2(I\times\Omega)}
\right).
\]
Thus we conclude the assertion (i) in Proposition \ref{thm:LW-cri}.
\end{proof}

\begin{proof}[Proof of {\rm (ii)} in Proposition \ref{thm:LW-cri}]
Let $I\subset \mathbb R_{+}$ be an interval, and let 
$u_1, u_2 \in S(I)$ 
satisfy the equation \eqref{eq.Duhamel2} with $u_1(0)=u_2(0)=u_0$. 
Then 
\[
\| u_{1} -u_{2} \|_{S(\tilde{I})}
\le C
\max\left\{\|u _{1}\|_{S(\tilde{I})}^{\frac{4}{d-2}}, \|u_{2}\|_{S(\tilde{I})}^{\frac{4}{d-2}}\right\}
\| u_{1} -u_{2} \|_{S(\tilde{I})}
\]
for any $\tilde{I}\subset I$ in the same way as \eqref{eq.s3-1}.
Choosing $\tilde{I}\subset I$ such that 
\[
 C
\max\left\{\|u _{1}\|_{S(\tilde{I})}^{\frac{4}{d-2}}, \|u_{2}\|_{S(\tilde{I})}^{\frac{4}{d-2}}\right\}
\le \frac12,
\]
we have $\| u_{1} -u_{2} \|_{S(\tilde{I})} \le 1/2 \| u_{1} -u_{2} \|_{S(\tilde{I})}$, which 
implies that $u_{1} \equiv u_{2}$ on $\tilde{I}$. 
By iterating this, we conclude that $u_{1} \equiv u_{2}$ on $I$. 
\end{proof}

To show the assertion (iii), we prepare the following:

\begin{lem}[Perturbation result]\label{prop:perturbation}
Let $T>0$ and
$v \in C([0,T); \dot H^1(L)) \cap S((0,T))$ be a solution to the equation
\[
\partial_t v + Lv = |v|^{\frac{4}{d-2}}v + e\quad \text{in }(0,T) \times \Omega
\]
with initial data $v(0) = v_0 \in \dot H^1(L)$, where $e=e(t,x)$ is a function on $(0,T) \times \Omega$.
Let $M>0$. Assume that  $v$ satisfies
\[
\| v \|_{L^\infty([0,T); \dot H^1(L))} +\| v \|_{S([0,T))} \le M.
\]
Then
there exist constants $\delta= \delta(M)>0$ and $C=C(M,\delta)>0$ such that
the following assertion holds{\rm :}
If $e \in L^2([0,T)\times \Omega)$ and $u_0 \in \dot H^1(L)$ satisfy
\[
\| e\|_{L^2([0,T)\times \Omega)} \le \delta,
\quad
\| u_0 - v_0\|_{\dot H^1(L)} \le \delta,
\]
then there exists a unique strong solution $u$ to \eqref{eq.NLH-cri} on $(0,T)\times \Omega$ with $u(0)=u_0$ such that
\[
\|u-v\|_{L^\infty([0,T); \dot H^1(L))} + \|u-v\|_{S((0,T))} \le C.
\]
\end{lem}

\begin{proof}
The proof is based on the method of proof of Theorem 2.14 in \cite{KM-2006}.
We consider the Cauchy problem
\begin{equation}\label{eq.appro}
\begin{cases}
\partial_t w + Lw = |v+w|^{\frac{4}{d-2}}(v+w) - |v|^{\frac{4}{d-2}}v -e\quad &\text{in }(0,T) \times \Omega,\\
w(0) = w_0 := u_0 - v_0 \in \dot H^1(L).
\end{cases}
\end{equation}
Define the map $\Phi$ by
\[
\Phi[w](t):=
e^{-tL}w_0 +
\int_0^t e^{-(t-s)L} \left(|v+w|^{\frac{4}{d-2}}(v+w) - |v|^{\frac{4}{d-2}}v -e\right)\,ds,
\]
and the complete metric space $X([0,T))$ by
\[
X([0,T)) := \left\{
w \in L^\infty([0,T); \dot H^1(L)) \cap S ((0,T)):
\| w \|_{X([0,T))} \le {K_0} \| w_0 \|_{\dot H^1(L)}
\right\},
\]
\[
\| w \|_{X([0,T))} :=
\| w \|_{L^\infty([0,T); \dot H^1(L))} +
\| w \|_{S((0,T))},
\]
where $K_0$ is the constant such that
\[
K_0 := \sum_{n=1}^{n_0} (4C_1)^n.
\]
Here $C_1$ is a positive constant in the estimate
\[
\|e^{-tL}f \|_{X([0,T'))}
\le C_1 \| f\|_{\dot H^1(L)}
\]
for any $0<T'\le T$ and $f\in  \dot H^1(L)$,
and $n_0 \in \mathbb N$ is determined later.
By Lemma \ref{lem:Str}, we estimate
\[
\begin{split}
& \|\Phi[w] \|_{X([0,T))}\\
\le &\, C_1\| w_0\|_{\dot H^1(L)}
+
C_2 \left(
\|v\|_{S((0,T))}^{\frac{4}{d-2}} + \|w\|_{S((0,T))}^{\frac{4}{d-2}}
\right)
\|w\|_{S((0,T))}
+ C_3 \| e\|_{L^2([0,T)\times \Omega)}
\end{split}
\]
for any $w \in X([0,T))$. Then,
choosing $\| w_0\|_{\dot H^1(L)}$ so small that
\[
C_2 \|w\|_{S((0,T))}^{\frac{4}{d-2}}\|w\|_{S((0,T))}
\le
C_2 K_0^{\frac{d+2}{d-2}} \| w_0 \|_{\dot H^1(L)}^{\frac{4}{d-2}} \| w_0 \|_{\dot H^1(L)}
\le C_1 \| w_0 \|_{\dot H^1(L)},
\]
we have
\begin{equation}\label{eq.small}
\|\Phi[w] \|_{X([0,T))}
\le 2 C_1\| w_0\|_{\dot H^1(L)}
+
C_2
\|v\|_{S((0,T))}^{\frac{4}{d-2}}
\|w\|_{S((0,T))}
+ C_3 \| e\|_{L^2([0,T)\times \Omega)}.
\end{equation}
Furthermore, splitting the interval $[0,T)$ into intervals $\{[T_{n-1}, T_n)\}_{n=1}^{n_0}$ such that
$T_0 = 0$, $T_{n_0} = T$ and
\[
\|w\|_{S((T_{n-1},T_n))} \le \| w_0\|_{\dot H^1(L)},
\]
\[
C_2 \|v\|_{S((T_{n-1},T_n))}^{\frac{4}{d-2}} \le C_1,
\]
\[
C_3 \| e\|_{L^2([T_{n-1},T_n)\times \Omega)} \le C_1\| w_0\|_{\dot H^1(L)}
\]
for $n=1,\ldots n_0$,
we find from \eqref{eq.small} that
\begin{equation}\label{eq.n1}
\|\Phi[w]\|_{X([T_0,T_1))}
\le 4 C_1\| w_0\|_{\dot H^1(L)}.
\end{equation}
Similarly, we get
\[
\|\Phi[w]\|_{X([T_1,T_2))} \le 4 C_1\| w(T_1)\|_{\dot H^1(L)} \le (4 C_1)^2\| w_0\|_{\dot H^1(L)},
\]
where we used \eqref{eq.n1} in the last step. Repeating this argument, we obtain
\[
\|\Phi[w]\|_{X([T_{n-1}, T_n))} \le (4 C_1)^n\| w_0\|_{\dot H^1(L)}
\]
for $n=1,\ldots n_0$. Hence
\[
\|\Phi[w]\|_{X([0,T))} \le \sum_{n=1}^{n_0} (4C_1)^n \| w_0\|_{\dot H^1(L)} = K_0 \| w_0\|_{\dot H^1(L)},
\]
which implies that $\Phi$ is a mapping from $X([0,T))$ into itself.
In a similar argument, we can prove that $\Phi$ is contractive from $X([0,T))$ into itself. Hence it follows from the fixed point argument that
there exists a unique solution $w$ to \eqref{eq.appro}, and
$u:= v+w$ is the required solution to \eqref{eq.NLH-cri}. Thus we conclude Lemma  \ref{prop:perturbation}.
\end{proof}

As a corollary, we have the following:

\begin{cor}\label{cor.maxtime}
Suppose that $L$ satisfies Assumption {\rm B}.
Let $u_0, u_{0,n} \in \dot H^1(L)$ for $n \in \mathbb N$, and let
$u$ and $u_n$ be solutions to \eqref{eq.NLH-cri} with $u(0)=u_0$ and $u_n(0) = u_{0,n}$, respectively.
If $u_{0,n} \to u_0$ in $ \dot H^1(L)$ as $n\to \infty$, then
\[
T_m(u_0) \le \liminf_{n\to\infty} T_m(u_{0,n})
\]
and
\[
\lim_{n\to\infty}\|u_n(t) - u(t)\|_{\dot H^1(L)} =0
\]
for any $t \in [0,T_m(u_0))$.
\end{cor}

\begin{proof}[Proof of {\rm (iii)} in Proposition \ref{thm:LW-cri}]
The assertion (iii) is an immediate consequence of Corollary \ref{cor.maxtime} and 
Lebesgue's dominated convergence theorem. 
\end{proof}

\begin{proof}[Proof of {\rm (iv)} in Proposition \ref{thm:LW-cri}]
Assume that $T_{m}=T_{m}(u_{0})<\infty$ and $\|u\|_{S((0,T_{m}))}<\infty$. 
Since 
\[
u(t_{0}+t) = e^{-tL}u(t) + 
\int_{0}^{t} e^{-(t-s)L}|u(t_{0}+s)|^{\frac{4}{d-2}}u(t_{0}+s)\,ds,
\]
it follows from Lemma \ref{lem:Str} that 
\begin{equation}\label{eq.D2}
\|e^{-tL}u(t_{0})\|_{S((0,T_{m}-t_{0}))}
\le \|u\|_{S((t_{0},T_{m}))} + 
C \|u\|_{S((t_{0},T_{m}))}^{\frac{d+2}{d-2}}.
\end{equation}
By the assumption $\|u\|_{S((0,T_{m}))}<\infty$, we may choose $t_{0}$ close enough 
to $T_{m}$ such that the right hand side of \eqref{eq.D2} is less than $\delta/2$, 
where $\delta$ is the constant given in Lemma~\ref{lem:local}. Then there exists 
$\varepsilon_{0}>0$ such that $\|e^{-tL}u(t_{0})\|_{S((0,T_{m}-t_{0}+\varepsilon_{0}))} < \delta$. Therefore, applying Lemma \ref{lem:local}, 
we extend the solution $u$ to the interval $[0, T_{m}+\varepsilon_{0})$.
This contradicts the maximality of $T_{m}$.
\end{proof}

\begin{proof}[Proof of {\rm (v)} in Proposition \ref{thm:LW-cri}]
For $u_{0}\in \dot H^{1}(L)$, the solution $u$ obtained in (i) satisfies
\[
\int_{0}^{t} 
(u_{s},u_{s})_{L^{2}(\Omega)}\, ds
+
\int_{0}^{t} 
(Lu,u_{s})_{L^{2}(\Omega)}\, ds
=
\int_{0}^{t} 
(|u|^{\frac{4}{d-2}}u,u_{s})_{L^{2}(\Omega)}\, ds
\]
for any $t\in(0,T_{m})$, where we note that 
$u_{s}, Lu, |u|^{\frac{4}{d-2}}u \in L^{2}([0,t]\times\Omega)$.
This implies the identity \eqref{eq.energy-identity}
by a straightforward calculation. 
The proof of (v) is finished. 
\end{proof}

\begin{proof}[Proof of {\rm (vi)} in Proposition \ref{thm:LW-cri}]
The proof is almost the same as in Lemma \ref{lem:local}. 
Indeed, instead of $X_{a,b}$ in the proof of Lemma \ref{lem:local}, 
define 
$X_{\varepsilon_0}:= \{ u \text{ on }\mathbb R_{+}\times \Omega : 
\| u \|_{S(I)}
\le \varepsilon_0\}$.
Then $(X_{\varepsilon_0}, \mathrm d)$ is a complete metric space with the distance $\mathrm d$ defined by \eqref{eq.distance}. 
Set the map $\Phi_{u_{0}}$ by \eqref{eq.Phi}. 
Now we have
\[
\| \Phi_{u_{0}}[u] \|_{S(\mathbb R_{+})}
\le \|e^{-tL}u_{0}\|_{S(\mathbb R_{+})} + C \varepsilon_0^{\frac{d+2}{d-2}}
\le \varepsilon_0,
\]
\[
\mathrm{d}(\Phi_{u_0}[u],\Phi_{u_0}[v])
\le C 
\varepsilon_0^{\frac{4}{d-2}}
\mathrm{d}(u,v)\le \frac12 \mathrm{d}(u,v),
\]
where we choose $\varepsilon_0 = 2\|e^{-tL}u_{0}\|_{S(\mathbb R_{+})}$ and 
$C \varepsilon_0^{\frac{4}{d-2}}\le 1/2$. 
Hence $\Phi_{u_{0}}$ is a contraction mapping from $X_{\varepsilon_0}$ into itself, and 
we find a unique element $u \in X_{\varepsilon_0}$ solving $\Phi_{u_{0}}[u]=u$ by the fixed point argument. By the uniqueness, this $u$ coincides with the corresponding solution in (i) of Proposition \ref{thm:LW-cri}. Therefore $u$ is definitely a global solution to \eqref{eq.NLH-cri} with  $\|u\|_{S(\mathbb R_{+})}\le 2\|e^{-tL}u_{0}\|_{S(\mathbb R_{+})}$ 
in the sense of Definition~\ref{def:2}. 
The proof of (vi) is finished.
\end{proof}

%%%%%%%%%%%%%%%% References %%%%%%%%%%%%%%%%

\end{document}